\documentclass[11pt,twoside]{amsart}
\usepackage{amssymb}
\usepackage{graphicx,color}
\newtheorem{theorem}{Theorem}[section]
\newtheorem{proposition}[theorem]{Proposition}

\newtheorem{lemma}[theorem]{Lemma}
\newtheorem{corollary}[theorem]{Corollary}
\newtheorem{definition}[theorem]{Definition}
\newtheorem{notation}[theorem]{Notation}
\newtheorem{example}[theorem]{Example}
\newtheorem{remark}[theorem]{Remark}

\newtheorem{problem}[theorem]{Problem}

\newcommand{\skipit}[1]{{}}
\newcommand{\prfend}{\hbox to7pt{\hfil}
\par\vskip-\baselineskip\hbox to\hsize
{\hfil\vbox {\hrule width6pt height6pt}}\vskip\baselineskip}

\newcommand{\ZZ}{\mathbb{Z}}

\newcommand{\N}{\mathbb{N}}
\newcommand{\Q}{\mathbb{Q}}
\newcommand{\TC}{\mathbb{T}}
\newcommand {\PP}{\mathbb{P}}

\newcommand{\af}{\mathbb{A}}

\newcommand{\myarrow}[2]{\hbox to #1pt{\hfil$\to$\hfil}{\hskip-#1pt{\raise
10pt\hbox to#1pt{\hfil$\scriptscriptstyle #2$\hfil}}}}

\begin{document}

\title{On the arithmetic Cohen-Macaulayness of  varieties parameterized by Togliatti systems.}

\author[Liena Colarte]{Liena Colarte}
\address{Department de matem\`{a}tiques i Inform\`{a}tica, Universitat de Barcelona, Gran Via de les Corts Catalanes 585, 08007 Barcelona,
Spain}
\email{liena.colarte@ub.edu}

\author[Emilia Mezzetti]{Emilia Mezzetti}
\address{Dipartimento di Matematica e  Geoscienze, Universit\`a di
Trieste, Via Valerio 12/1, 34127 Trieste, Italy}
\email{mezzette@units.it}

\author[Rosa M. Mir\'o-Roig]{Rosa M. Mir\'o-Roig}
\address{Department de matem\`{a}tiques i Inform\`{a}tica, Universitat de Barcelona, Gran Via de les Corts Catalanes 585, 08007 Barcelona,
Spain}
\email{miro@ub.edu}

\begin{abstract} 
Given any diagonal cyclic subgroup $\Lambda \subset GL(n+1,k)$ of order $d$, let $I_d\subset k[x_0,\ldots, x_n]$ be the ideal generated by all monomials $\{m_{1},\hdots, m_{r}\}$ of degree $d$ which are invariants of $\Lambda$. $I_d$ is a monomial Togliatti system, provided $r \leq \binom{d+n-1}{n-1}$, and in this case the  projective toric variety $X_d$ parameterized by $(m_{1},\hdots, m_{r})$  is called a $GT$-variety with group $\Lambda$. We prove that all these $GT$-varieties are arithmetically Cohen-Macaulay and we give a combinatorial expression of their Hilbert functions. In the case $n=2$, we compute explicitly the Hilbert function, polynomial and series of $X_d$.  We determine a minimal free resolution of its homogeneous ideal and we show that it is a binomial prime ideal generated by quadrics and cubics. We also provide the exact number of both types of generators. Finally, we pose the problem of determining whether a surface parameterized by a Togliatti system is aCM. We construct examples that are aCM and examples that are not.

\end{abstract}

\thanks{Acknowledgements:   The first and third authors  are partially   supported
by  MTM2016--78623-P. 
The second author is supported by PRIN 2017SSNZAW\_005 ''Moduli theory and birational classification'', by FRA of the University of Trieste, and is a member of  INDAM -- GNSAGA
}

\maketitle

\tableofcontents

\markboth{}{}

\today

%*****************************************************************************
\large

\section{Introduction.}

In 1946 \cite{T2}, Eugenio Togliatti classified the rational surfaces of $\mathbb P^N$, $N\geq 5$, parameterized by cubics and representing a Laplace equation of order $2$, i.e., whose osculating spaces have all dimension strictly less than the expected $5$. Only for one of the surfaces found by Togliatti the apolar ideal to the ideal generated by the polynomials giving the parameterization is artinian, and it is the ideal $J=(x^3, y^3, z^3, xyz)\subset K[x,y,z]$. In 2007 \cite{BK}, Brenner and Kaid proved that, over an algebraically closed field of  characteristic $0$, $J$ is the only ideal of the form $(x^3, y^3, z^3, f(x,y,z))$, with $f\in k[x,y,z]$ homogeneous of degree $3$, failing the weak Lefschetz property (see Section \ref{defs and prelim results}, \ref{Lefschetz properties and Togliatti systems}, for the definition). In 2013, the connection between these two examples has been clarified and extended. In the article \cite{MM-RO}, it is proved that, given an artinian ideal $I \subset k[x_0,\hdots,x_n]$ generated by $r$ forms of degree $d$, if $r \leq \binom{n+d-1}{n-1}$, then $I$ fails the weak Lefschetz property  in degree $d-1$ if and only if the $n$-dimensional variety $Y$ parameterized by the forms of degree $d$ apolar to $I$ satisfies a Laplace equation of order $d-1$. These ideals $I$, now called {\em Togliatti systems}, have been studied in a series of articles, see 
\cite{AMRV}, \cite{CMM-RS}, \cite{CMM-RS1}, \cite{CM-R}, \cite{DMMMN}, \cite{MM-R1}, \cite{MM-R}, \cite{MkM-R} and \cite{M-RS}.
In \cite{MM-R1} and \cite{M-RS} there are descriptions of the minimal monomial Togliatti systems with \lq\lq\thinspace low'' number of generators, where minimal means that it does not contain any smaller Togliatti system. 

There is an interesting family of examples generalizing one aspect of the ideal $J$ found by Togliatti. More precisely, we consider the following situation. We fix integers $2 \leq n < d$, $0 \leq \alpha_{0} \leq \cdots \leq \alpha_{n} < d$ such that $GCD(\alpha_{0},\hdots,\alpha_{n},d) = 1$ and we fix $e$, a $d$th primitive root of $1$. Let $\Lambda \subset GL(n+1,k)$ be the cyclic subgroup of order $d$ generated by the diagonal matrix $M_{d; \alpha_{0},\hdots, \alpha_{n}}:= diag(e^{\alpha_0},\hdots, e^{\alpha_n})$. We denote by $I_{d}$ the  artinian ideal generated by all monomials $\{m_{1},\hdots, m_{r}\}$ of degree $d$ which are invariants of $\Lambda$ and by $X_{d}$ the image of the morphism $\varphi_{I_{d}}: \PP^{n} \to \PP^{r-1}$ defined by $(m_{1},\hdots, m_{r})$. With this notation, $J$ is the ideal corresponding to $\Lambda = \langle M_{3;0,1,2} \rangle \subset GL(3,k)$.
The study of the ideals $I_{d} \subset k[x_0,x_1,x_2]$ started in \cite{MM-R}, where it is also determined the geometry of the surface $S_{d}$ corresponding to $\Lambda = \langle M_{d;0,1,2} \rangle \subset GL(3,k)$. 
The minimal free resolution of $S_{d}$ is described, as well as it is proved that $S_{d}$ is an arithmetically Cohen-Macaulay surface generated by quadrics and cubics. Afterwards in \cite{CM-R}, some results are generalized for the threefold $F_{d}$ corresponding to $\Lambda = \langle M_{d;0,1,2,3} \rangle$. The minimality of the ideals $I_{d}$ for any group $\Lambda = \langle M_{d;\alpha_0, \alpha_1, \alpha_2} \rangle$ is established in \cite{CMM-RS} and \cite{DMMMN}, and the argument relies on a careful study of the permanent of certain circulant matrices. 

In the present paper, we focus our attention on the arithmetic Cohen-Macaulay property (shortly aCM) of any variety $X_{d}$, as well as surfaces parameterized by Togliatti systems $I \subset k[x_0,x_1,x_2]$.
All these varieties are monomial projections of Veronese varieties. Any result in this direction should therefore be considered as a contribution to the longstanding problem of deciding whether projections of Veronese varieties are aCM, posed by Gr\"obner in \cite{Grob}. Our first result is Theorem \ref{Theorem: invariant basis of GT-systems}, stating the non-trivial fact that any monomial invariant of $\Lambda$ of degree a multiple of $d$  can be expressed as a product of monomial invariants  of $\Lambda$ of degree $d$. It relies on a result of Erd\"os, Ginzburg and Ziv (\cite{EGZ}). By a $GT$-system we shall mean a Togliatti system $I \subset k[x_0,\hdots,x_n]$ whose associated morphism $\varphi_{I}: \PP^{n} \to \PP^{r-1}$ is a Galois covering with group $\ZZ/d\ZZ$. It follows that $I_{d}$ is a $GT$-system with group $\Lambda$, provided $r \leq \binom{d+n-1}{n-1}$, and in this case we call $X_{d}$ a $GT$-variety with group $\Lambda$. 

Our main result proves that any variety $X_{d}$ is aCM, and so $GT$-varieties with group $\Lambda$ are aCM (Theorem \ref{Theorem: GT-varieties are aCM}). We deduce it from Theorem  \ref{Theorem: invariant basis of GT-systems}, proving that the coordinate ring of $X_d$ is the ring  of invariants $R^{\overline{\Lambda}}$, where $\overline{\Lambda}$ is the diagonal linear group of order $d^2$ generated by $M_{d;\alpha_{0},\hdots, \alpha_{n}}$ and $M_{d;1,\hdots,1}= diag(e,\hdots,e)$. Afterwards, we turn our attention to the Hilbert function of $X_{d}$ and we give a combinatorial description of it. In the case $n = 2$, we are able to obtain Theorem \ref{Proposition: Hil function surfaces} containing an explicit expression for the Hilbert polynomial and series, as well as  a minimal free resolution of any $GT$-surface (Theorem \ref{Theorem: minimal free resolution}). From this we provide a complete description of the homogeneous ideal of any $GT$-surface.

Finally, we address the general problem of the arithmetic Cohen-Macaulayness of surfaces parameterized by monomial Togliatti systems  whose coordinate rings are not rings of invariants of  finite linear groups. We give a counterexample showing that this property is not true in general. However, we provide a new class of Togliatti systems, whose varieties are aCM. These are not GT-systems, but are obtained as a different generalization of the ideal $J$. The proof relies on the study of the associated numerical semigroup, using a criterion due to Goto and Watanabe in \cite{GW} and  Trung in \cite{Trung}.

\vspace{0.3cm}
Let us outline how this work is organized. Section \ref{defs and prelim results} contains the basic definitions and results needed in the rest of this paper. We introduce semigroup rings and the rings of invariants by finite groups. Next, we present the basic facts on Galois coverings and quotient varieties by finite groups of automorphisms. Finally, we recall the notion of Togliatti systems and $GT$-systems introduced in  \cite{CMM-RS}, \cite{MM-R} and \cite{MM-RO}.

The main results of this paper are collected in Sections \ref{the aCM of GT-varieties} and \ref{GT-surfaces}. In Section \ref{the aCM of GT-varieties} we prove that any variety $X_{d}$ is aCM. In Section \ref{GT-surfaces}, we focus on the geometric properties of $GT$-surfaces. We explicitly determine their Hilbert function, polynomial and series. Fixed an integer $d\geq 3$ and $\Lambda = \langle M_{d;0,a,b} \rangle \subset GL(3,k)$ with $0 < a < b$, we are able to find a function $\theta(a,b,d)$ such that, for all $t \geq 0$, the Hilbert function $HF(X_{d},t)$ of $X_{d}$ equals $\frac{dt^{2} + \theta(a,b,d)t + 2}{2}$
% where $\theta(a,b,d) \geq 3$ is an integer we explicitly determine 
(see Theorem \ref{Proposition: Hil function surfaces}). We find a minimal free resolution of any $GT$-surface (Theorem \ref{Theorem: minimal free resolution}), which allows us to conclude that its homogeneous ideal is a binomial prime ideal minimally generated  by quadrics and cubics. We give the exact number of both types of generators (see Corollary \ref{Corollary: number of generators}).

Section \ref{aCM togliatti surfaces} concerns the arithmetic Cohen-Macaulayness of surfaces parameterized by monomial Togliatti systems whose coordinate rings are not rings of invariants of  finite linear groups. 

\vskip 4mm \noindent
 {\bf Acknowledgements.} This work was partially carried out while the first author was visiting the {\em Universit\`{a} degli Studi di Trieste}. The first author would like to thank the university for its hospitality, specially the {\em Dipartimento di Matematica e Geoscienze}. The authors gratefully thank the
referees of this paper for their useful suggestions and comments. The authors would also like to thank M. Salat for useful discussions on $GT$-systems and related topics.

\vskip 4mm \noindent {\bf Notation.} Throughout this paper, $k$ denotes an algebraically closed field of characteristic zero, $R = k[x_{0},\hdots,x_{n}]$ and $GL(n+1,k)$ the multiplicative group of invertible $(n+1) \times (n+1)$ matrices with coefficients in $k$. If $z,z'$ are positive integers, we denote by $(z,z')$ the greatest common divisor of $z$ and $z'$.

%%%%%%%%%%%%%%%%%%%%%%%%%%%%%%%%%%%%%%%%%%%%%%%%%%
\section{Preliminaries.}
\label{defs and prelim results}

In this section, we introduce the main objects and results we shall use.  First, we define semigroups and normal semigroups, and we present three results on the Cohen-Macaulayness of semigroup rings needed in the sequel (see  \cite{Bruns-Herzog}, \cite{GW}, \cite{Hochster} and \cite{Trung-Hoa}). Second, we prove that quotient varieties by the action of finite groups of automorphisms are Galois coverings and  we translate this result from the point of view of Invariant Theory. For a further exposition in Invariant Theory of finite groups, see for instance \cite{Bruns-Herzog} and \cite{Stanley}. Finally, we introduce the weak Lefschetz property  and the  notions of  Togliatti systems and $GT$-systems.

\subsection{Semigroup rings and rings of invariants}\label{semigroup rings and rings of Invariant}
%\noindent{\bf Semigroup rings and rings of invariants.}

By a {\em semigroup}, we mean a finitely generated subsemigroup $H = \langle h_{1}, \hdots, h_{t} \rangle$ of $\ZZ^{n+1}$. We denote by $L(H)$ the additive subgroup of $\ZZ^{n+1}$ generated by $H$ and by $r$ the rank of $L(H)$ in $\ZZ^{n+1}$. We also denote by $k[H] \subseteq R$ the semigroup ring associated to $H$, i.e., the graded $k$-algebra whose basis elements correspond to the monomials $X^{h_{j}}$, $j = 1,\hdots,t$, where $X^{h_j}$ denotes the monomial $x_0^{a_0}\cdots x_n^{a_n}$ with $h_{j} =(a_0,\ldots ,a_n)$. By a basis of $k[H]$ we mean a set of elements $\theta_{1},\dotsc,\theta_{\ell}\in k[H]$ such that $k[H] = k[\theta_{1},\dotsc,\theta_{\ell}]$.

\begin{definition} \rm  \label{Definition: Normal semigroup} A semigroup $H \subset \ZZ^{n+1}$ is called {\em normal} if it coincides with its saturation $\overline{H}:= \{w \in L(H) \, \mid \, zw \in H, \; \text{for some} \; z \in \ZZ_{\geq 0}\}$.
\end{definition}

Concerning normal semigroups, Hochster proves the following result.

\begin{proposition}\label{Proposition:Hochster CM} If a semigroup $H$ is normal, then $k[H]$ is Cohen-Macaulay.
\end{proposition}
\begin{proof} See \cite[Theorem 1]{Hochster}.
\end{proof}

A large family of normal semigroups comes from Invariant Theory, precisely those associated to finite abelian groups acting linearly on $R$. We take $\Lambda =
\ZZ/\ZZ d_{1} \oplus \cdots \oplus \ZZ/\ZZ d_{r}$ and we  choose $d_{i}$-th primitive roots of unity $e_{i}$, $i = 1, \hdots, r$.  $\Lambda$ can be linearly represented in $GL(n+1,k)$ by means of $r$ diagonal matrices $diag(e_{i}^{u_{0,i}}, \hdots, e_{i}^{u_{n,i}})$, where $u_{j,i} \in \N$, $0 \leq j \leq n$, $1 \leq i \leq r$. We consider the ring of invariants $R^{\Lambda} := \{p \in R \,\mid\, \lambda(p) = p \; \text{for all} \; \lambda \in \Lambda\}$. A polynomial $p \in R^{\Lambda}$ if and only if all its monomials belong to $R^{\Lambda}$. By Noether's degree bound (see \cite[1.2 Theorem.]{Stanley}), $R^{\Lambda}$ has a finite basis consisting of monomials of degree at most the order of $\Lambda$.
Let $X^{h_{1}}, \hdots, X^{h_{t}}$ be a monomial basis of $R^{\Lambda}$ and $H = \langle h_{1}, \hdots, h_{t} \rangle$. Then $R^{\Lambda} \cong k[H]$.
Furthermore, a monomial $x_{0}^{a_{0}} \cdots x_{n}^{a_{n}} \in R^{\Lambda}$ if and only if $(a_{0}, \hdots, a_{n})$ satisfies the system of congruences:
\begin{equation}
a_{0}u_{0,i} + \cdots + a_{n}u_{n,i} \equiv 0 \pmod{d_i},\  i = 1,\hdots, r.
\end{equation}
Now, if $w \in L(H)$ is such that $zw \in H$ for some $z \in \ZZ_{\geq 0}$, then  $w \in H$. So $H$ is normal and $k[H]$ is a CM ring.

By \cite[Proposition 13]{Hochster-Eagon}, the ring of invariants of any finite group acting linearly on $R$ is CM. This is a particular case of \cite[Proposition 12]{Hochster-Eagon} that we present next.
Let $A$ be a subring of $R$: a {\em Reynolds operator} is a $A$-linear map $\rho: R \to A$ such that $\rho_{|A} = id_{A}$. We have:

\begin{theorem}\label{Hochster-Eagon} Let $A$ be a subring of $R$ such that there exists a Reynolds operator $\rho$ and $R$ is integral over $A$. Then $A$ is a Cohen-Macaulay ring.
\end{theorem}
\begin{proof}
See \cite[Proposition 12]{Hochster-Eagon}.
\end{proof}

Let $G \subset GL(n+1,k)$ be a finite group acting on $R$. We denote by $R^{G}$ the ring of invariants of $G$. One can easily check  that the map $\rho:R\to R^G$, defined by $\rho(p) = |G|^{-1} \sum_{g \in G} g(p)$, is a Reynolds operator. Furthermore, any element $p \in R$ is a solution of the equation
$$\prod_{g \in G} (Y - g(p)) = 0,$$
which is a polynomial in $Y$ with coefficients in $R^{G}$. So $R$ is integral over $R^{G}$ and, by Theorem \ref{Hochster-Eagon}, $R^{G}$ is CM.

Partially motivated by the results of Proposition \ref{Proposition:Hochster CM} and Theorem  \ref{Hochster-Eagon},  Goto, Suzuki and Watanabe, and Trung proved:

\begin{theorem}\label{Theorem:Criterion aCM-Trung} Let $H$ be a semigroup and  assume that  there exist $\Q$-linearly independent elements $f_{1}, \hdots, f_{m} \in H$ such that $z \cdot H \subset \langle f_{1}, \hdots, f_{m}  \rangle$, for some positive integer $z$. The following conditions are equivalent.
\begin{itemize}
\item[(i)] $k[H]$ is Cohen-Macaulay.
\item[(ii)] If $w \in L(H)$ and there exist $i,j$ with $1\leq i\leq j\leq m$, such that $w + f_{i} \in H$ and $w + f_{j} \in H$, then $w \in H$.
\item[(iii)] $\cap_{i=1}^{m} (f_{i} + H) \subset (\sum_{i=1}^{m} f_{i}) + H$.
\item[(iv)] $H = \cap_{i=1}^{m} H_{i}$, where $H_{i} = \{w \in L(H) \, \mid \, w + g \in H \;\text{for some}\; g \in (\sum_{j = 1, j \neq i}^{m} \Q_{+}f_{j}) \cap H$.
\end{itemize}
In particular, set $H^{1} = \{w \in \overline{H} \, \mid \,
w + f_{i}, w + f_{j} \in H \; \text{for some} \; i \neq j \in \{1,\hdots, m\}\}$. Then $k[H]$ is Cohen-Macaulay if and only if $H^{1} = H$.
\end{theorem}
\begin{proof} See \cite[Theorem 2.6]{GW} and \cite[Lemma 2]{Trung}.
\end{proof}

\begin{remark}\rm Let $H$ be a normal semigroup which satisfies the hypothesis of Theorem \ref{Theorem:Criterion aCM-Trung}. By Proposition \ref{Proposition:Hochster CM},
the semigroup ring $k[H]$ is CM. Notice that $H$ trivially verifies Theorem \ref{Theorem:Criterion aCM-Trung}(ii).
\end{remark}

%\noindent{\bf Galois coverings and quotients varieties.}
\subsection{Galois coverings and quotient varieties}\label{Galois coverings and quotient varieties} 

We recall that a {\em covering} of a variety $X$ consists of a variety $Y$ and a finite morphism $f: Y \to X$. The {\em  group of deck transformations $G:=Aut(f)$} is defined to be the group of automorphisms of $Y$ commuting with $f$. We say that $f: Y \to X$ is a {\em covering with group} $Aut(f)$. 

\begin{definition} \rm  A covering $f: Y \to X$ with group $Aut(f)$ is {\em Galois} if $Aut(f)$ acts transitively on a fibre $f^{-1}(x)$ for some $x \in X$.
\end{definition}

When a group $G$ acts on a variety $X$, there is a natural way of constructing Galois coverings.

\begin{definition} \rm  Let $G$ be a group acting on a variety $X$. The {\em quotient of $X$ by $G$} is defined to be a variety $Y$ with
a surjective morphism $p: X \to Y$ such that any morphism
$\rho: X \to Z$ to a variety $Z$ factors through $p$
if and only if $\rho(x) = \rho(g(x))$, for all $x \in X$ and
$g \in G$.
\end{definition}

\begin{remark} \rm If it exists, the quotient variety is unique up to isomorphism and is denoted by $X/G$. In particular, the morphism $p: X \to X/G$ verifies that if $x,y \in X$, then $p(x) = p(y)$ if and only if $g(x) = y$, for some $g \in G$.
\end{remark}

\begin{proposition}\label{Proposition: Quotient variety affine} Let $G$ be a finite group acting on an affine variety $X$. Then, $X/G$ is the affine variety whose coordinate ring $A(X/G)$ is the ring of regular functions on X, invariants of $G$, and $\pi: X \to X/G$ is the quotient of $X$ by $G$.
\end{proposition}
\begin{proof}
See \cite[Section 12, Proposition 18]{Serre1}.
\end{proof}

\begin{proposition}\label{Proposition: Quotient variety projective} Let $G$ be a finite group acting on a projective variety $X$ and $X/G$ its quotient space. If the orbit of any point $x \in X$ is contained in an affine open subset of $X$, then $X/G$ is a projective variety and $\pi: X \to X/G$ is the quotient of $X$ by $G$.
\end{proposition}
\begin{proof}
See \cite[Section 12, Proposition 19]{Serre1}.
\end{proof}

\begin{proposition}\label{Prop:Preliminaries} Let $X$ be a projective variety and $G \subset Aut(X)$ be a finite group.
If the quotient variety $X/G$ exists, then $\pi: X \to X/G$ is a Galois covering with group $G$.
\end{proposition}

\begin{proof} Set $G = \{g_{1},\hdots, g_{n},id\}$. The group $Aut(\pi)$ consists of all automorphisms of $X$ commuting with $\pi$. If $f:X \to X$ belongs to $Aut(\pi)$, then for all $x \in X$ we have $\pi(f(x)) = \pi(x)$. For any $x \in X$, there exists $g_{i} \in G$ such that $f(x) = g_{i}(x)$, and hence $X = V(f-g_{1}) \cup \cdots \cup V(f-g_{n})$. The irreducibility of $X$ allows us to conclude that $f = g_{i}$, for some $g_{i} \in G$. Therefore, $Aut(\pi) = G$ and it is clear that given $\pi(x) \in X/G$, the fibre $\pi^{-1}(\pi(x)) = G_{x}$, so $Aut(\pi) = G$ acts transitively on $\pi^{-1}(\pi(x))$.
\end{proof}

A finite group of automorphisms of the affine space $\mathbb A^{n+1}$ can be regarded as a finite  group $G \subset GL(n+1,k)$ acting on R. Let $\{f_{1},\hdots, f_{t}\}$ be a basis of $R^{G}$, also called a set of {\em fundamental invariants of $G$},  and let $k[w_{1},\hdots, w_{t}]$ be  the polynomial ring in the new variables $w_{1},\hdots, w_{t}$.  We denote by $syz(f_1,\hdots,f_t)$ the kernel of the morphism from  $A^{n+1}$ to $A^{t}$ defined by $w_{i} \to f_{i}$, $i = 1,\hdots, t$. We have:

\begin{proposition}\label{Proposition: qv by fg acting linearly on poly ring} Let $G \subset GL(n+1,k)$ be a finite group acting on $\mathbb{A}^{n+1}$, let $\{f_1,\hdots,f_t\}$ be a set of fundamental invariants of $G$ and let $\pi: \af^{n+1} \to \af^{t}$ be the morphism defined by $(f_1,\hdots,f_t)$. Then,
\begin{itemize}
\item[(i)] $\pi(\af^{n+1})$ is the quotient of $\mathbb{A}^{n+1}$ by $G$ with affine coordinate ring $R^{G}$.
\item[(ii)] $R^{G} \cong k[w_{1},\hdots, w_{t}]/syz(f_{1},\hdots, f_{t})$, i.e., $I(\pi (\mathbb{A}^{n+1}))=syz(f_{1},\hdots, f_{t}). $
\item[(iii)] $\pi$ is a Galois covering of $\pi(\af^{n+1})$ with group $G$.
\end{itemize}
\end{proposition}
\begin{proof} See \cite[Section 6]{Stanley}, Proposition \ref{Proposition: Quotient variety affine} and Proposition \ref{Prop:Preliminaries}.
\end{proof}

The cardinality of a general orbit $G(a)$, $a \in \af^{n+1}$, is called the degree of the covering. Moreover, if we can find a homogeneous set of fundamental invariants $\{f_1,\hdots, f_t\}$ of $G$ such that $\pi: \PP^{n} \to \PP^{t-1}$ is a morphism, then  the projective version of Proposition \ref{Proposition: qv by fg acting linearly on poly ring} is true.

%\vspace{0.3cm}
%\noindent{\bf Lefschetz properties and Togliatti systems.}
\subsection{Lefschetz properties and Togliatti systems}\label{Lefschetz properties and Togliatti systems}
Let $I\subset R $ be a homogeneous artinian ideal.  The {\em weak Lefschetz property} (WLP for short) is an important property of these ideals, which has attracted much interest in the last years, see for instance \cite{BK}, \cite{LlibreWatanabe}, \cite{MM-RO}, \cite{Migliore-MR-Nagel}, \cite{Migliore-MR-Nagel1} and \cite{Migliore-Nagell}. We recall the definition. We say that $I$ {\em has the WLP}
if there is a linear form $L \in R_1$ such that, for all
integers $j$, the multiplication map
\[
\times L: (R/I)_{j} \to (R/I)_{j+1}
\]
has maximal rank. We say that $I$ {\em fails the WLP in degree $j_0$} if for any linear form $L \in R_{1}$, the multiplication map $\times L: (R/I)_{j_0} \to (R/I)_{j_{0}+1}$ has not  maximal rank. In 2013 \cite{MM-RO}, Mezzetti, Mir\'{o}-Roig and Ottaviani established a close connection between algebraic and geometric language showing that the failure of the WLP for  ideals generated by forms of the same degree    is related to the existence of varieties whose all osculating spaces of a certain order  have dimension less than  expected.  To state the precise statement, we shortly recall the definition of   the Macaulay's inverse system $I^{-1}$ of  $I$ and the language of osculating spaces and Laplace equations. 

In addition to $R$, we consider a second polynomial ring $\mathcal{R}=k[X_0,\ldots,X_n]$. We have the apolarity action of $R$ on $\mathcal{R}$ by partial differentiation, i.e., if $F\in R$ and $h\in \mathcal{R}$, then $F\cdot h=F(\frac{\partial}{\partial X_0}, \ldots,\frac{\partial}{\partial X_n})\circ h$.  By definition, the Macaulay inverse system  $I^{-1}$ of a graded ideal $I\subset R$ is the graded $R$-submodule of $\mathcal{R}$ annihilator of $I$: $I^{-1}=\{h\in\mathcal{R}\mid F\cdot h=0 \  \hbox{for all}\  F\in I\}$.
On the geometric side, we recall that, if $X$ is a rational projective variety with a birational parameterization
$\PP^n \dashrightarrow X \subset \PP^{r-1}$ given by $r$ forms $F_1, \ldots, F_r$ of degree $d$ in
$R$, then
the projective $s$th osculating space $\TC_x^{(s)}X$, for $x$ general, is generated by the $s$-th partial derivatives
of $F_1, \cdots ,F_r$ at the point $x$.  The expected dimension of $\TC_x^{(s)}X$
is $max\{r-1, \binom{n+s}{s}-1\}$, but
it could be lower. If strict inequality
holds for all smooth points of $X$, and $\dim
\TC_x^{(s)}X=\binom{n+s}{s}-1-\delta$ for general $x$,
then $X$ is said to satisfy $\delta$ Laplace equations of order $s$.
Indeed, in this case the partials of order $s$ of $F_1,\ldots,F_r$
are linearly dependent, which gives $\delta$ differential equations
of order $s$ satisfied by $F_1,\ldots,F_r$.

In \cite{MM-RO}  the following theorem is proved.

\begin{theorem} \label{tea} Let $I\subset R=k[x_0\ldots, x_n]$ be an artinian
ideal
generated
by $r$ forms $F_1,\dotsc,F_{r}$ of degree $d$ and let $I^{-1}$ be its Macaulay inverse system.
If
$r\le \binom{n+d-1}{n-1}$, then
  the following conditions are equivalent.
\begin{itemize}
\item[(i)] $I$ fails the WLP in degree $d-1$;
\item[(ii)]  $F_1,\dotsc,F_{r}$ become
$k$-linearly dependent on a general hyperplane $H$ of $\PP^n$;
\item[(iii)] The $n$-dimensional variety $Y:=\overline{\varphi(\PP^{n})}$,
where
$\varphi=\varphi_{I^{-1}} \colon\PP^n \dashrightarrow \PP^{\binom{n+d}{d}-r-1}$ is the  rational map associated to $(I^{-1})_d$,
  satisfies at least one Laplace equation of order
$d-1$.
\end{itemize}
\end{theorem}
\begin{proof} See \cite[Theorem 3.2]{MM-RO}.
\end{proof}

An artinian ideal $I \subset R$ generated by $r \leq \binom{d+n-1}{n-1}$ forms of degree $d$ defines a \emph{Togliatti system}
if it satisfies the three equivalent conditions in  Theorem \ref{tea}.  In particular, a Togliatti system is called {\em smooth} if the variety $Y$ in Theorem \ref{tea}(iii) is smooth, and {\em monomial} if $I$ can be generated by monomials. The name is in honour of  Eugenio Togliatti, who proved that for
$n = 2$ the only smooth
Togliatti system of cubics is the monomial ideal 
\begin{equation}\label{3togliatti} I = (x_0^3,x_1^3,x_2^3,x_0x_1x_2)\subset k[x_0,x_1,x_2]\end{equation}
(see \cite{BK}, \cite{MM-R}, \cite{T1} and \cite{T2}). The corresponding variety $Y$, parameterized by  $(I^{-1})_3$, is a smooth surface in $\mathbb P^5$, known as Togliatti surface; its $2$-osculating spaces have all dimension $\leq 4$ instead of the expected dimension $5$. The systematic study of Togliatti systems $I$ was initiated in \cite{MM-RO}, where one can find in particular a classification  of monomial Togliatti systems with \lq\lq\thinspace low'' number of generators;   for further results the reader can see \cite{AMRV}, \cite{MM-R1}, \cite{MM-R}, \cite{MkM-R} and \cite{M-RS}. In  \cite{MM-R} the authors introduced the notion of Galois-Togliatti system (shortly \emph{GT-system}), which we recall now.

\begin{definition} \rm \label{Defi:GT}
  A \emph{GT-system} is a Togliatti system $I_{d} \subset R$ generated by $r$ forms $F_{1},\dotsc,F_{r}$ of degree $d$
  such that the morphism $\varphi_{I_{d}}\colon\PP^{n}\rightarrow \PP^{r-1}$ defined by $(F_{1},\dotsc,F_{r})$ is a Galois covering with cyclic group $\ZZ /d\ZZ$.
\end{definition}

In the sequel, the image of the morphism $\varphi_{I_d}$ will be denoted by $X_d$. The varieties $X_d$ and $Y$, introduced in Theorem \ref{tea} are called apolar. The first example of GT-system is the ideal (\ref{3togliatti}). The corresponding pair of apolar varieties is formed by the Togliatti surface $Y\subset\mathbb P^5$ and the cubic surface $X_3\subset\mathbb P^3$.

\begin{example} \rm Fix integers $n = 2$, $d = 5$, fix $e$ a $5$th primitive root of $1$ and let $\Lambda = \langle diag(1,e,e^{3}) \rangle \subset GL(3,k)$ be a cyclic group of order $5$. The homogeneous component of degree $5$ of  $R^{\Lambda}$ is generated by the invariant monomials $x_{0}^5,x_{1}^{5}, x_{2}^{5}, x_{0}^{2}x_{1}^{2}x_{2}, x_{0}x_{1}x_{2}^{3}$. In total we have $r = 5$ monomials so the inequality $r \leq \binom{n+d-1}{n-1}$ is satisfied. One proves that the ideal $I_{5} \subset R$ generated by these monomials fails the WLP in degree $4$ and the morphism $\varphi_{I_{5}}: \PP^{2} \to \PP^{4}$ is a Galois covering of degree $5$ with cyclic group $\ZZ/5\ZZ$ (see Corollary \ref{Corollary: GT-systems}). Actually $\varphi_{I_{5}}(\PP^{2})$ is the quotient surface by the action of the finite group of automorphisms of $\PP^{2}$ generated by $diag(1,e,e^3)$.
\end{example}

In the following, we will study GT-systems $I_d$ generated by forms of degree $d$ which are invariants of a finite diagonal cyclic subgroup of $GL(n+1,K)$ of order $d$. Note that Definition \ref{Defi:GT} does not assume that the ideal is monomial. For examples of non-monomial Togliatti systems, the reader can look at \cite{CMM-RS1}. However, the Togliatti systems we will study in Sections \ref{the aCM of GT-varieties}, \ref{GT-surfaces} and \ref{aCM togliatti surfaces} are all monomial.
%%%%%%%%%%%%%%%%%%%%%%%%%%%%%%%%%%%%%%%%%%%%%%%%%%%%%%%%%%%%%%%%%

\section{The arithmetic Cohen-Macaulayness of GT-varieties.}
\label{the aCM of GT-varieties}

In this section, we study the ideals generated by all monomials $\{m_{1},\hdots, m_{\mu_{d}}\}$ of degree $d$ which are invariants of a finite diagonal cyclic group $\Lambda \subset GL(n+1,k)$ of order $d$. They are monomial GT-systems, provided $\mu_{d} \leq \binom{d+n-1}{n-1}$. We study the varieties  associated to them, which we call $GT$-varieties with group $\Lambda$; in particular we prove that they are aCM.

To this end, we fix integers $2 \leq n < d$ and $0 \leq \alpha_0 \leq \cdots \leq \alpha_n < d$ with $GCD(\alpha_0,\hdots,\alpha_n,d) = 1$. We denote by $M_{d;\alpha_0,\hdots, \alpha_n}$ the diagonal matrix $diag(e^{\alpha_0},\hdots, e^{\alpha_n})$, where  $e$ is a $d$th primitive root of $1$. We consider the cyclic group $\Lambda = \langle M_{d;\alpha_0,\hdots,\alpha_n} \rangle \subset GL(n+1,k)$ of order $d,$ and the abelian group $\overline{\Lambda} \subset GL(n+1,k)$  of order $d^2$ generated by $M_{d;\alpha_0,\hdots,\alpha_n}$ and $M_{d;1,\hdots,1} = diag(e,\hdots,e)$.
As usual $R^{\Lambda}$ (respectively $R^{\overline{\Lambda}}$) represents the ring of invariants of $\Lambda$ (respectively $\overline{\Lambda}$).
Let $\{m_{1},\hdots, m_{\mu_{d}}\}$ be the set of all monomials of degree $d$ which are invariants of $\Lambda$ and  denote by $I_{d}$ the monomial artinian ideal generated by them. Let $\varphi_{I_{d}}: \PP^{n} \to \PP^{\mu_{d}-1}$ be the morphism associated to $I_{d}$ and define $X_{d}:=\varphi_{I_{d}}(\PP^{n})$. Let $w_{1},\hdots,w_{\mu_{d}}$ be a new set of indeterminates, let $S: = k[w_1,\hdots,w_{\mu_{d}}]$ denote the polynomial ring and $I(X_{d}) \subset S$  the homogeneous ideal of $X_{d}$.

Our first result shows that $\{m_{1},\hdots,m_{\mu_{d}}\}$ is a $k$-algebra basis of $R^{\overline{\Lambda}}$, i.e., $R^{\overline{\Lambda}} = k[m_1,\hdots, m_{\mu_{d}}]$. This will allow us to prove that any variety $X_{d}$ is aCM and that $I_{d}$ is a monomial $GT$-system, provided $\mu_{d} \leq \binom{d+n-1}{n-1}$.

\begin{theorem} \label{Theorem: invariant basis of GT-systems} The set of monomials of degree $d$ which are invariants of $\Lambda$ is a $k$-algebra basis of $R^{\overline{\Lambda}}$.
\end{theorem}
\begin{proof}
We want to prove that $R^{\overline{\Lambda}} = k[m_1,\hdots,m_{\mu_{d}}]$. Since $\overline{\Lambda}$ acts diagonally on $R$, this is equivalent to show that for all $t \geq 1$, any monomial $m \in R^{\overline{\Lambda}}$ of degree $td$ belongs to $k[m_{1},\hdots,m_{\mu_{d}}]$, i.e., it is a product of $t$ monomials $m_{i_{1}},\hdots, m_{i_{t}} \in \langle m_{1},\hdots,m_{\mu_{d}}\rangle$, non necessarily different. We proceed by induction on $t$. We fix $t \geq 2$, we take a monomial $m = x_{0}^{a_0}x_{1}^{a_1}\cdots x_{n}^{a_n} \in R^{\Lambda}$  of degree $td$ and we consider  $\mathcal{S} := \{\alpha_{0}, \overset{a_0}{\hdots}, \alpha_{0}, \alpha_{1}, \overset{a_1}{\hdots}, \alpha_{1}, \hdots, \alpha_{n}, \overset{a_n}{\hdots}, \alpha_{n}\}$
 a sequence of integers where  $\alpha_0$ is repeated $a_0$ times, $\alpha_1$ is repeated $a_1$ times, and so on. Since $t \geq 2$,  \; $\mathcal{S}$ contains more than $2d-1$ elements. Hence by \cite[Theorem]{EGZ} and \cite{GG}, there exists a subsequence $\mathcal{S}' \subset \mathcal{S}$ of $d$ elements summing to a multiple $rd$ of $d$. We write $\mathcal{S'}= \{\alpha_{0}, \overset{b_0}{\hdots}, \alpha_{0}, \alpha_{1}, \overset{b_1}{\hdots}, \alpha_{1}, \hdots, \alpha_{n}, \overset{b_n}{\hdots}, \alpha_{n}\}$, and we consider the monomial $\overline{m} = x_{0}^{b_{0}}x_{1}^{b_{1}}\cdots x_{n}^{b_{n}} \in R$. Clearly $\overline{m}$ divides $m$. Moreover, $b_{0}+b_{1}+\cdots + b_{n} = d$ and $\alpha_0 b_0 + \alpha_1 b_1 + \cdots + \alpha_n b_n = rd$. Therefore, $\overline{m}$ is an invariant of $\Lambda $,  and $m/\overline{m} \in k[m_{1},\hdots,m_{\mu_{d}}]$ by induction hypothesis. So the proof is complete.
\end{proof}

\begin{example} \rm \label{Example: Togliatti exam aCM} We illustrate Theorem \ref{Theorem: invariant basis of GT-systems} with the example of ideal (\ref{3togliatti}).
Fix $n = 2$, $d = 3$ and let $\Lambda = \langle M_{3;0,1,2} \rangle \subset GL(3,k)$. A monomial
 $x_{0}^{a_{0}}x_{1}^{a_{1}}x_{2}^{a_{2}} \in R^{\overline{\Lambda}}$ if and only if there exist integers $t \geq 1$ and $r \in \{0,1,2,\hdots, 2t\}$ such that $(a_{0},a_{1},a_{2}) \in \ZZ_{\geq0}^{3}$ is a solution of the system
 $$(*)_{t,r} = \left\{\begin{array}{rcl}
 a_{0} + a_{1} + a_{2} &=& 3t\\
 a_{1} + 2a_{2} &=& 3r.
 \end{array}\right.$$
In particular, $\{x_{0}^{3}, x_{1}^{3}, x_{2}^{3}, x_{0}x_{1}x_{2}\}$ is the set of all monomials of degree $3$ in $R^{\Lambda}$. Fix $t > 1$ and let $m = x_{0}^{a_{0}}x_{1}^{a_{1}}x_{2}^{a_{2}} \in R^{\overline{\Lambda}}$ be a monomial of degree $3t$. First we assume that $a_{0}a_{1}a_{2} \neq 0$. We may also assume that $a_{0} = min\{a_{0},a_{1},a_{2}\}$, the other cases follow in the same way. Then clearly $m=(x_{0}x_{1}x_{2})^{a_{0}}x_{1}^{a_{1}-a_{0}}x_{2}^{a_{2}-a_{0}}$ and $x_{1}^{a_{1}-a_{0}}x_{2}^{a_{2}-a_{0}} \in R^{\overline{\Lambda}}$. So we have that $a_{1} - a_{0} + a_{2} - a_{0}$ and $a_{1} - a_{0} + 2(a_{2} - a_{0})$ are multiples of $3$, which implies  that $a_{1}-a_{0}$ and $a_{2}-a_{0}$ are multiples of $3$.
Now we assume $a_{0}a_{1}a_{2} = 0$. We may  suppose that $a_{0} = 0$ and $a_{1}a_{2} \neq 0$. We have that $a_{1} + a_{2}$ and $a_{1} + 2a_{2}$ are multiples of $3$, which gives that $a_{1}$ and $a_{2}$ are multiples of $3$.
\end{example}

\begin{theorem} \label{Theorem: GT-varieties are aCM} 
$X_{d}$ is a toric aCM variety.
\end{theorem}

\begin{proof}
By definition, $X_d$ is parameterized by monomials and hence it is toric. By Theorem \ref{Theorem: invariant basis of GT-systems}, we have that $\{m_{1},\hdots,m_{\mu_{d}}\}$ is a set of fundamental invariants of $\overline{\Lambda}$. Therefore, the theorem follows directly from the projective version of Proposition \ref{Proposition: qv by fg acting linearly on poly ring}(i) and \cite[Proposition 13]{Hochster-Eagon}. 
\end{proof}

\begin{corollary}\label{Corollary: GT-systems} If $\mu_{d} \leq \binom{n+d-1}{n-1}$, then $I_{d}$ is a monomial $GT$-system.
\end{corollary}
\begin{proof}
We have to prove that $I_{d}$ is a Togliatti system and $\varphi_{I_{d}}:\PP^{n} \to \PP^{\mu_{d}-1}$ is a Galois covering with group $\ZZ/d\ZZ$.
By
Theorem \ref{Theorem: invariant basis of GT-systems} and the projective version of Proposition \ref{Proposition: qv by fg acting linearly on poly ring}, $\varphi_{I_{d}}:\PP^{n} \to \PP^{\mu_{d}-1}$ is a Galois covering with group $\ZZ/d\ZZ$. 
It only remains to prove that 
 if $\mu_{d} \leq \binom{d+n-1}{n-1}$, then $I_{d}$ fails the WLP in degree $d-1$. By \cite[Proposition 2.2]{Migliore-MR-Nagel} and Theorem \ref{tea} this is equivalent to check that
 for $L = x_0+\cdots +x_n \in R_{1}$, the  map $\times L: (R/I_{d})_{d-1} \to (R/I_{d})_{d}$ is not injective.  We take $p = \prod_{j=1}^{d-1} (e^{j\alpha_{0}}x_0 + \cdots + e^{j\alpha_{n}}x_n)$. It is straightforward to see that $\times L(p) = \prod_{j=0}^{d-1} (e^{j\alpha_{0}}x_0 + \cdots + e^{j\alpha_{n}}x_n)$ is an invariant of $\Lambda$, so $\times L(p) = 0$ and $\times L$ is not injective.
\end{proof}

\begin{definition} \rm \label{Defi:GTgroup}
An  ideal $I_d$  as in Corollary \ref{Corollary: GT-systems} is called a {\em $GT$-system with group $\Lambda$}.
\end{definition}

We present examples of families of monomial $GT$-systems, which also motivates our next definition.
\begin{example}\rm \label{Example: family of examples}
(i) Fix integers $d\geq 3$ and $0 < a < b$.  Let $\Lambda = \langle M_{d;0,a,b} \rangle \subset GL(3,k)$. In \cite{MM-R} the authors prove that $\mu_{d} \leq d+1$. Hence, by Corollary \ref{Corollary: GT-systems}, $I_{d}$ is a monomial $GT$-system.

(ii) Fix integers $3 = n < d$ and let $\Lambda = \langle M_{d;0,1,2,3} \rangle \subset GL(4,k)$.
In \cite{CM-R} it is proved that $\mu_{d} \leq \binom{2+d}{2}$. So by Corollary \ref{Corollary: GT-systems}, $I_{d}$ is a monomial $GT$-system.

(iii) Fix an integer $n\geq 2$ and let $\Lambda$ be the subgroup of  $GL(n+1,k)$ generated by $M_{n+1;0,1,2, \hdots, n}.$
In \cite{CMM-RS}, the authors show that $\mu_{n+1} \leq \binom{2n}{n-1}$. By Corollary \ref{Corollary: GT-systems}, the associated ideal $I_{n+1}$ is a monomial $GT$-system. 
\end{example}

\begin{definition} \rm We call  {\em $GT$-variety with group $\Lambda$}  any projective variety $\varphi _{I_d}(\PP^n)$ associated to a
 a $GT$-system $I_{d}$  with group $\Lambda = \langle M_{d;\alpha_0,\hdots,\alpha_n} \rangle \subset GL(n+1,k)$. 
\end{definition}

Example \ref{Example: family of examples}(iii) provides us with examples of $GT$-varieties of any dimension $n \geq 2$. As a corollary of Theorem \ref{Theorem: GT-varieties are aCM} we have:
\begin{corollary} Any $GT$-variety $X_{d}$ with group $\Lambda = \langle M_{d;\alpha_0,\hdots,\alpha_n} \rangle \subset GL(n+1,k)$ is aCM.
\end{corollary}

%%%%%%%%%%%%%%%%%%%%%%%%%%%%%%%%%%%%%%%%%%%%%%%%%%%%%%%%%%%%%%%%%
\section{Hilbert function of GT-surfaces.}
\label{GT-surfaces}

In this section, we give a combinatorial description of the Hilbert function of any $GT$-variety $X_{d}$ with group $\Lambda = \langle M_{d;\alpha_0,\hdots,\alpha_n} \rangle \subset GL(n+1,k)$ in terms of the invariants of $\Lambda$. For the particular case of  $GT$-surfaces, we explicitly compute their Hilbert function, polynomial and series. We also determine a minimal free resolution of their homogeneous ideals. As a corollary, we obtain that the homogeneous ideal of any $GT$-surface is minimally generated by quadrics and cubics.

The following well-known result is needed.

\begin{lemma}\label{Corollary: Hilbertfunction GT} Let $G \subset GL(n+1,k)$ be a finite group and fix $t \geq 1$. We have:
$$dim (R^{G})_{t} = \frac{1}{|G|} \sum_{g \in G} trace(g^{(t)})$$
where $g^{(t)}$ is the linear map induced by $g$ on $R_{t}$.
\end{lemma}
\begin{proof} See \cite[Theorem 2.1]{Stanley}.
\end{proof}

\begin{remark} \rm  Let $G \subset GL(n+1,k)$ be a finite group and let $\{m_{1},\hdots, m_{L}\}$ be  a monomial basis of $R_d$. Fix $g \in G$ and $t \geq 1$. In this basis, the linear map $g^{(t)}$ is represented by a matrix whose columns are the coordinates of $g(m_{i})$, $i = 1,\hdots, L$. In particular, if $G$ acts diagonally on $R$, then $g^{(t)}$ is represented by a diagonal matrix.
\end{remark}

The following proposition follows from  \cite[Theorem 6.4.2]{Bruns-Herzog}. For sake of completeness we include an elementary proof.
\begin{proposition}\label{Theorem: Hil func GT-variety}  The Hilbert function $HF(X_{d},t)$ of $X_{d}$ in degree $t \geq 1$ equals the number of monomials of degree $td$ which are invariants of $\Lambda$.
\end{proposition}
\begin{proof} Fix $t \geq 1$ and let $m_{1},\hdots, m_{N} \in R$ be all monomials of degree $td$; we write $m_{i} = x_{0}^{a_{0}^{i}}\cdots x_{n}^{a_{n}^{i}}$, $i = 1,\hdots, N$. By Lemma \ref{Corollary: Hilbertfunction GT} we have the equalities:
$$HF(X_{d},t) = dim((R^{\Lambda})_{td})
= \frac{1}{d} \sum_{\lambda \in \Lambda} trace(\lambda^{(td)})
= \frac{1}{d} trace(\sum_{\lambda \in \Lambda} \lambda^{(td)}).$$
Fix $j \in \{1,\hdots, d-1\}$ and $\lambda = M_{d;\alpha_0,\hdots,\alpha_n}^{j} \in \Lambda$. We can represent the induced linear map $\lambda^{(td)}$ by a diagonal matrix whose entry in position $(i,i)$, we note $\lambda^{(td)}_{(i,i)}$, corresponds to $e^{\alpha_{0}a_{0}^{i} + \cdots
+ \alpha_{n}a_{n}^{i}}$, $i = 1,\hdots, N$. If $m_{i} \in R^{\Lambda}$, then $\lambda^{(td)}_{(i,i)} = 1$. Otherwise $\lambda^{(td)}_{(i,i)} = e^{j(\alpha_{0}a_{0}^{i} + \cdots + \alpha_{n}a_{n}^{i})} \neq 1$. Now determining $trace(\sum_{\lambda \in \Lambda} \lambda^{(td)})$ is straightforward.  Indeed, the $(i,i)$ entry of the matrix $\sum_{\lambda \in \Lambda} \lambda^{(td)})$ is $d$ if $m_{i}\in R^{\Lambda}$, and equal to $1 + e^{j(\alpha_{0}a_{0}^{i} + \cdots + \alpha_{n}a_{n}^{i})} + e^{2j(\alpha_{0}a_{0}^{i} + \cdots + \alpha_{n}a_{n}^{i})} + \cdots + e^{(d-1)j(\alpha_{0}a_{0}^{i} + \cdots + \alpha_{n}a_{n}^{i})}$ otherwise. If $\xi \neq 1$ is a $d$th root of $1$, we have $1 + \xi + \cdots + \xi^{d-1} = 0$, and the result follows.
\end{proof}

For fixed $t \geq 1$, the monomials of degree $td$ in $R^{\Lambda}$ are completely determined by the following systems:
$$(*)_{t,r} = \left\{ \begin{array}{lclclclcl}
y_{0} &+& y_{1} &+& \cdots &+& y_{n} &=& td\\
\alpha_{0}y_{0} &+& \alpha_{1}y_{1} &+& \cdots &+& \alpha_{n}y_{n} &=& rd
\end{array}\right. , \quad  \ r = 0, \hdots, \alpha_{n}t.$$
For each $r \in \{0,\hdots, \alpha_{n}t\}$, we define $|(*)|_{t,r}$ to be the number of solutions of $(*)_{t,r}$ in $\ZZ_{\geq 0}^{n+1}$. We can rewrite Proposition \ref{Theorem: Hil func GT-variety} as follows.
\begin{corollary}\label{Corollary: Counting num invariants} For any $t \geq 1$, we have: $HF(X_{d},t)
= \sum_{r=0}^{\alpha_{n}t} |(*)_{t,r}|$.
\end{corollary}

\begin{example} \label{Example: Cubic surface} \rm Continuing with Example \ref{Example: Togliatti exam aCM}, we consider $\Lambda = \langle M_{3;0,1,2} \rangle \subset GL(3,k)$. The monomials of degree $3$ in $R^{\Lambda}$ are $\{x_{0}^3, x_{1}^3, x_{2}^3, x_{0}x_{1}x_{2}\}$. Next we list those of degree $3t$, for $t = 2,3,4$.
$$\begin{array}{rl}
t = 2, & \{x_{0}^6,x_{0}^3x_{1}^3, x_{0}^4x_{1}x_{2},x_{1}^6,x_{0}x_{1}^4x_{2},x_{0}^2x_{1}^2x_{2}^2, x_0^3x_2^3,x_0x_1x_2^4,x_2^6\},\;  \underline{HF(X_{3},2) = 10.}\\
t = 3, & \{x_0^9, x_0^6 x_1^3, x_0^7x_1x_2, x_0^3 x_1^6, x_0^4 x_1^4 x_2, x_0^5 x_1^2 x_2^2, x_0^6 x_2^3, x_1^9,
 x_0 x_1^7 x_2, x_0^2 x_1^5 x_2^2, x_0^3 x_1^3 x_2^3, x_0^4 x_1 x_2^4, x_1^6 x_2^3,\\
        &  x_0 x_1^4 x_2^4,
 x_0^2 x_1^2 x_2^5, x_0^3 x_1^6, x_1^3 x_2^6, x_0 x_1 x_2^7, x_2^9\},\; \underline{HF(X_{3},3) = 19.}\\
t = 4, & \{x_0^{12}, x_0^9 x_1^3, x_0^{10} x_1 x_2, x_0^6 x_1^6, x_0^7 x_1^4 x_2, x_0^8 x_1^2 x_2^2, x_0^9 x_2^3,
 x_0^3 x_1^9, x_0^4 x_1^7 x_2, x_0^5 x_1^5 x_2^2, x_0^6 x_1^3 x_2^3, x_0^7 x_1 x_2^4, \\
 &  x_1^{12},
 x_0 x_1^{10} x_2, x_0^2 x_1^8 x_2^2, x_0^3 x_1^6 x_2^3, x_0^4 x_1^4 x_2^4, x_0^5 x_1^2 x_2^5,
 x_0^6 x_2^6, x_1^9 x_2^3, x_0 x_1^7 x_2^4, x_0^2 x_1^5 x_2^5, x_0^3 x_1^3 x_2^6, x_0^4 x_1 x_2^7, \\
 &  x_1^6 x_2^6, x_0 x_1^4 x_2^7, x_0^2 x_1^2 x_2^8, x_0^3 x_2^9, x_1^3 x_2^9, x_0 x_1 x_2^{10}, x_2^{12}\}, \; \underline{HF(X_{3},4) = 31}.\\
\end{array}$$
Let $w_{1},w_{2},w_{3},w_{4}$ be new indeterminates, we denote by $S = k[w_{1},w_{2},w_{3},w_{4}]$ the polynomial ring. $X_{3}$ is the cubic surface $V(w_{1}w_{2}w_{3}-w_{4}^{3}) \subset \PP^{3}$ and we have $HP(X_{3})(t) = \frac{3}{2}t^2 + \frac{3}{2}t + 1$.
\end{example}

In Theorem \ref{Theorem: GT-varieties are aCM}, we proved that $S/I(X_{d})$ is CM; moreover, since $X_d$ is toric, we have that its ideal is generated by binomials: $I(X_{d}) = (w_{1}^{\delta_{1}}\cdots w_{\mu_{d}}^{\delta_{\mu_{d}}} - w_{1}^{\gamma_{1}} \cdots
w_{\mu_{d}}^{\gamma_{\mu_{d}}} \,\mid \, m_{1}^{\delta_{1}}\cdots m_{\mu_{d}}^{\delta_{\mu_{d}}}  = m_{1}^{\gamma_{1}} \cdots
m_{\mu_{d}}^{\gamma_{\mu_{d}}}, \; \sum_{i=1}^{\mu_{d}} \delta_{i} = \sum_{i=1}^{\mu_{d}} \gamma_{i})$. 
We now consider a minimal graded free $S$-resolution $N_{\bullet}$ of $S/I(X_{d})$.
$$N_{\bullet}: \quad 0 \to N_{\mu_{d}-n-1} \to \cdots \to N_{2} \to N_{1} \to S \to S/I(X_{d}) \to 0,$$
where $N_{l} \cong \bigoplus_{j \geq l}^{f_l} S(-j-l)^{b_{l,j}}$ and  $b_{l,f_l} > 0$, $1 \leq l \leq \mu_{d}-n-1$.

As usual, the {\em Cohen-Macaulay type of $S/I(X_{d})$} is  the dimension of the free $S$-module $N_{\mu_{d}-n-1}$. We recall that $S/I(X_{d})$ is {\em level} if  $N_{\mu_{d}-n-1}$ is generated in only one degree and that $S/I(X_{d})$ is {\em Gorenstein} if
it is level and $\dim (N_{\mu_{d}-n-1}) = 1$. We denote by $reg(X_{d}) := f_{\mu_{d}-n-1}+1$ the Castelnuovo-Mumford regularity of $S/I(X_{d})$.
The ideal $I(X_{d})$ is minimally generated by $b_{1,j}$ binomials of degree $j+1$, $j = 1,\hdots,f_1$.  We set $i = min\{1 \leq j \leq f_1 \,\mid\, b_{1,j} \neq 0\}$. We highlight two combinatorial ways of computing $b_{1,i}$ which follow from Proposition \ref{Theorem: Hil func GT-variety}. For completeness we include a simple proof. Let $\{m_{1}^{t},\hdots, m_{N}^{t}\} \subset R^{\Lambda}$ be the set of all  monomials of degree $td$. Each $m_{j}^{t}$ is a product of $t$ monomials of degree $d$ in $R^{\Lambda}$ (see Theorem \ref{Theorem: invariant basis of GT-systems}). We denote by $|m_{j}^{t}|$ the number of different ways of expressing $m_j^t$ as product of $t$ monomials of degree $d$.

\begin{proposition} \label{Proposition:qj} With the above notation, we have:
$$b_{1,i} = \binom{\mu_{d}+i}{i+1} - \sum_{r=0}^{(i+1)\alpha_{n}}|(*)|_{i+1,r} = \sum_{j=1}^{N} (|m_{j}^{i+1}| - 1).$$
\end{proposition}
\begin{proof} Computing the Hilbert function of $X_{d}$ in degree $i+1$ from $N_{\bullet}$, we obtain that $HF(X_{d},i+1) = dim_{k}(S_{i+1})-b_{1,i}$. By Corollary \ref{Corollary: Counting num invariants}, we get $dim_{k}(S_{i+1}) - b_{1,i} = \sum_{r=0}^{\alpha_{n}(i+1)d} |(*)_{i+1,r}|$ which implies the first equality. By Proposition \ref{Proposition:qj}, $\, b_{1,i} = \binom{\mu_{d}+i}{i+1} - \sum_{r=0}^{(i+1)\alpha_{n}d}|(*)|_{i+1,r}$. Now $\binom{\mu_{d}+i}{i+1}$ is the number of all possible combinations of $i+1$ monomials of degree $d$ in $R^{\Lambda}$. Thus $\binom{\mu_{d}+i}{i+1} = \sum_{j=1}^{N} |m_{j}^{i+1}|$, from which the second equality  follows.
\end{proof}

\begin{example} \rm (i) In the case of the cubic surface $X_{3}$ of Example \ref{Example: Cubic surface}, $HF(X_{3},1) = 4$, $HF(X_{3},2)= 10$ and $HF(X_3,3) = 19$. We obtain $b_{1,1} = \binom{4+1}{2} - 10 = 0$ and $b_{1,2} = \binom{4+2}{3} - 19 = 20-19 = 1$.

(ii) Let $\Lambda = \langle M_{4;0,1,2,3} \rangle \subset GL(4,k)$ (see Example \ref{Example: family of examples}(ii)). In \cite[Example 4.2]{CM-R}, the authors compute a minimal set of binomial generators of the associated $GT$-variety $X_{4}$. They show that $I(X_{4})$ is generated by exactly $12$ quadrics. On the other hand, we have $HF(X_4,1) = 10$ and $HF(X_4,2) = 43$. By Proposition \ref{Proposition:qj},
$b_{1,1} = \binom{10 + 1}{2} - 43 = 55-43= 12$ which confirms \cite[Example 4.2]{CM-R}.
\end{example}

From now on we focus on $GT$-surfaces. We fix an integer $d\geq 3$ and  a cyclic group $\Lambda = \langle M_{d;0,a,b} \rangle \subset GL(3,k)$ of order $d$ with $0 < a < b$. From Example \ref{Example: family of examples}(i) it follows that the ideal $I_{d}$ generated by all monomials $\{m_{1},\hdots, m_{\mu_{d}}\} \subset R^{\Lambda}$ of degree $d$ is a monomial $GT$-system with group $\Lambda$, so the associated variety $X_{d}$ is a $GT$-surface with group $\Lambda$. In the rest of this section we will use the following notation.

\begin{notation}\rm \label{notation a'} We  put
$$a' = \frac{a}{(a,d)}, \ b' = \frac{b}{(b,d)}, \ d' = \frac{d}{(a,d)}, \ d'' = \frac{d}{(b,d)}.$$
We denote by $\lambda$ and $\mu$  the uniquely determined integers such that $0 < \lambda \leq d'$ and $b = \lambda a' + \mu d'$.
\end{notation}

By Proposition \ref{Theorem: Hil func GT-variety}, $HF(X_{d},t)$ is the number of integer solutions $(y_{0},y_{1},y_{2}) \in \ZZ_{\geq 0}^{3}$ of the systems
$$(*)_{t,r} = \left\{\begin{array}{lclclcl}
y_{0} &+& y_{1} &+& y_{2} &=& td\\
& & ay_{1} &+& by_{2} &=& rd\\
\end{array}\right. ,\quad r = 0,\hdots, bt$$
or, equivalently,

\begin{lemma} \label{Lemma: (**) solutions} $HF(X_{d},t)$ equals the number of integer solutions $(y_{0},y_{1},y_{2}) \in \ZZ_{\geq 0}^{3}$ of the systems:
$$(**)_{t,r} = \left\{ \begin{array}{lclclcl}
y_{0} &+& y_{1} &+& \frac{y_{2}}{(a,d)} &=& td\\
& & y_{1} &+& \lambda \frac{y_{2}}{(a,d)} &=& rd'
\end{array}\right. ,\quad r = 0,\hdots, t\lambda.$$
which satisfy $y_{1} + y_{2} \leq td$.
\end{lemma}
\begin{proof} Let $(y_{0},y_{1},y_{2}) \in \ZZ_{\geq 0}^{3}$ be a solution of
$(*)_{t,r}$ for some $r \in \{0,\hdots, bt\}$. Notice that $(a,d)$ divides $y_{2}$, since $((a,d),b) = 1$ and $((a,d),a) = ((a,d),d) = (a,d)$. We have
$ay_{1} + by_{2} = ay_{1} + a'\lambda y_{2} + \mu d' y_{2} = rd$. For convenience we write $y_{2}' = \frac{y_{2}}{(a,d)}$. Therefore,
$a'y_{1} + a'\lambda y_{2}' = (r - \mu y_{2}') d'$ which implies that
$a'$ divides $(r-\mu y_{2}')$. We obtain $y_{1} + \lambda y_{2}' = r'd'$, where $0 \leq r' \leq \lambda t$. Thus, $(y_{0},y_{1},y_{2})$ uniquely induces a solution of the systems $(**)_{t,r}$ satisfying $y_{1} + y_{2} \leq td$.

Conversely, let $(y_{0},y_{1},y_{2}')$ be a solution of $(**)_{t,r}$ for some $r \in \{0,\hdots, t\lambda\}$ such that $y_{1} + (a,d)y_{2}' \leq td$. 
We have that $y_{1} + \lambda y_{2}' = rd'$, which implies $ay_{1} + a\lambda y_{2}' = ra'd$. Since $a'\lambda = b - \mu d'$, we get $ay_{1} + a\lambda y_{2}'= ay_{1} + b(a,d)y_{2}' - \mu d'(a,d)y_{2}' = ra'd$
and so
$ay_{1} + b(a,d)y_{2}' = (ra' + \mu y_{2}')d$. Writing $y_{2}:= (a,d)y_{2}'$, \; $(y_{0},y_{1},y_{2})$ verifies that $ay_{1} + by_{2} = r'd$ for some $0 \leq r' \leq tb$. Then  $(y_{0},y_{1},y_{2})$ induces a unique solution of some system $(*)_{t,r}$ if and only if $y_{1} + y_{2} \leq td$.
\end{proof}

\begin{example} \rm \label{Example:mu GT systems} (i) Consider $\Lambda = \langle M_{8;0,3,5} \rangle \subset GL(3,k)$  and write $5 = 3\cdot 7 + (-2)\cdot8$. Both systems $(*)_{1,r}$ and $(**)_{1,r}$ give the same set of monomials:
$$\{x_{0}^8, x_{0}^6x_1x_2, x_0^4 x_1^2 x_2^2, x_1^8, x_0^2 x_1^3 x_2^3, x_1^4 x_2^4, x_2^8\}.$$

(ii) Consider $\Lambda = \langle M_{6;0,2,3} \rangle \subset GL(3,k)$. The systems $(*)_{1,r}$ give the set of  seven monomials:
$$x_0^6, x_0^3 x_1^3, x_0^4 x_2^2, x_1^6, x_0 x_1^3 x_2^2, x_0^2 x_2^4, x_2^6.$$
The solutions $(y_{0},y_{1},y_{2}) \in \ZZ_{\geq 0}^{3}$ of the systems
$$(**)_{1,r} = \left\{ \begin{array}{lclclcl}
y_{0} &+& y_{1} &+& y_{2} &=& 6\\
& & y_{1} &+& 3y_{2} &=& 3r
\end{array}\right. ,\quad r = 0,1,2,3,$$
are: $(6,0,0)$, $(3,3,0)$, $(5,0,1)$, $(0,6,0)$, $(2,3,1)$, $(4,0,2)$, $(1,3,2),$ $(3,0,3)$, $(0,3,3)$, $(2,0,4),$ $ (1,0,5)$ and $(0,0,6)$, but only the following seven triples $(6,0,0)$, $(3,3,0)$, $(5,0,1)$, $(0,6,0)$,  $(2,3,1)$, $(4,0,2)$, $(3,0,3)$ satisfy also  $y_{1} + 2y_{2} \leq 6$, according to Lemma \ref{Lemma: (**) solutions}.
\end{example}

\begin{remark} \rm \label{Remark: lambda bigger than (a,d)} \begin{itemize}
\item[(i)]  Assume $(a,d) = 1$ (respectively $(b,d)=1$) and write $b = \lambda a + \mu d$ (respectively $a = \lambda'b + \mu'd)$. It is straightforward to check $\lambda \neq 1$ (respectively $\lambda' \neq 1$).
\item[(ii)] Assume $(a,d),(b,d) > 1$. If $(a,d) < (b,d)$ (respectively $(b,d) < (a,d)$), it is easy to see that we can write $b = \lambda a' + \mu d'$ with $(b,d) < \lambda$ (respectively $a = \lambda'b' + \mu'd'')$ with $(a,d) < d''$).
\end{itemize}
\end{remark}

\begin{theorem}\label{Proposition: Hil function surfaces} Using Notation \ref{notation a'}, let $\theta(a,b,d) := (a,d) + (\lambda,d') + (\lambda - (a,d),d')$. Then,
\begin{enumerate}
\item[(i)] $HF(X_{d},t) = \frac{d}{2}t^{2} + \frac{1}{2}\theta(a,b,d)t + 1;$
\item[(ii)] 
\begin{displaymath}  HS(X_{d},z) =\frac{\frac{d-\theta(a,b,d)+2}{2} z^2 + \frac{d + \theta(a,b,d) - 4}{2}z + 1}{(1-z)^{3}}.\end{displaymath}
\end{enumerate}
\end{theorem}

\begin{proof}
(i) By Lemma \ref{Lemma: (**) solutions}, we only have to count the number of solutions $(y_{0},y_{1},y_{2}) \in \ZZ_{\geq 0}^{3}$ of $(**)_{t,r}$, $r = 0,\hdots, t\lambda$, which satisfy $y_{1} + (a,d)y_{2} \leq td$. Without loss of generality, we may assume that $(a,d) < (b,d)$. Fix $r \in \{0,\hdots, t\lambda\}$. The solutions of $(**)_{t,r}$ are determined by the values of $y_2$ such that
$$\max\{0, \lceil \frac{(r-t(a,d))d'}{\lambda -1} \rceil\} \leq
y_{2} \leq \lfloor \frac{rd'}{\lambda} \rfloor,$$ and are of the form $(td-rd' + (\lambda - 1)y_{2}, rd' - \lambda y_{2}, y_{2})$. Now we impose $y_{1} + (a,d)y_{2} \leq td$. This is equivalent to $rd' - \lambda y_{2} \leq td - (a,d)y_{2}$ if and only if $(\lambda - (a,d))y_{2} \geq
rd' - td$. Thus we have to count the number of $y_{2}$'s in the range
$\max\{0, \lceil \frac{(r-(a,d)t)d')}{\lambda - (a,d)}\} \rceil \leq y_{2} \leq \lfloor \frac{rd'}{\lambda} \rfloor$. Putting all together, we get:
$$HF(X_{d},t) = 2 + \sum_{r = 1}^{t\lambda-1} (\lfloor \frac{rd'}{\lambda} \rfloor +1) - \sum_{r = t(a,d)+1}^{t\lambda -1} (\lceil \frac{(r-(a,d)t)d'}{\lambda - (a,d)} \rceil + 1).$$

Given two positive integers $m,n$, it holds that $\sum_{i=1}^{n-1} \lfloor \frac{im}{n}\rfloor  = \frac{(m-1)(n-1) + (m,n) - 1}{2}$. So

$$HF(X_{d},t) = 2 + t\lambda - 1 + \frac{(td'-1)(t\lambda  -1) + t(d',\lambda)-1}{2}$$
$$ - (\sum_{r=1}^{t(\lambda - (a,d)) -1} \lceil \frac{rd't}{(\lambda - (a,d))t}\rceil) - (t(\lambda - (a,d)) - 1).$$
We observe that $\lceil \frac{rd't}{(\lambda - (a,d))t} \rceil = \lfloor \frac{rd't}{(\lambda - (a,d))t} \rfloor$ if and only if $rd'$ is a multiple of $\lambda - (a,d)$; otherwise $\lceil \frac{rd't}{(\lambda - (a,d))t} \rceil = \lfloor \frac{rd't}{(\lambda - (a,d))t} \rfloor + 1$. We consider the set $\mathcal S=\{r\in \ZZ \mid 1 \leq r \leq t(\lambda - (a,d) -1) \ \hbox{and} \ t(\lambda - (a,d))\ \hbox{divides} \ rd't\}$. An integer $r\in \mathcal S$ if and only if $rd'$ is a multiple of
$LCM(d',\lambda - (a,d)) = \frac{d'(\lambda - (a,d))}{(\lambda - (a,d),d')}$. So $|\mathcal S| = t(\lambda - (a,d), d') - 1$ and we obtain:
$$\sum_{r=1}^{t(\lambda - (a,d)) -1} \lceil \frac{rd't}{(\lambda - (a,d))t}\rceil = \frac{(td'-1)(t\lambda - t(a,d) - 1)}{2} + t(\lambda - (a,d)) - 1 - t(d',\lambda - (a,d)).$$
It is straightforward to check that
\begin{equation}\label{Equation Hil}
HF(X_{d},t) = \frac{d}{2}t^2 + \frac{((a,d) + (d',\lambda) + (d',\lambda - (a,d)))}{2}t + 1.
\end{equation}

(ii) By definition $HS(X_{d},z) = \sum_{t \geq 0} HF(X_{d},t)z^{t}=$

$$ = \sum_{t \geq 0} \frac{d}{2}t^{2}z^{t} + \sum_{t \geq 0} \frac{\theta(a,b,d)}{2}tz^{t} + \sum_{t \geq 0}z^{t}=$$
$$ = \frac{\frac{d}{2} z(z+1)}{(1-z)^{3}} +
\frac{\frac{\theta(a,b,d)}{2}z}{(1-z)^{2}} + \frac{1}{1-z} = \frac{\frac{d-\theta(a,b,d)+2}{2} z^2 + \frac{d + \theta(a,b,d) - 4}{2}z + 1}{(1-z)^{3}}.$$
\end{proof}

As a direct consequence of the above computations and the fact that $S/I(X_{d})$ is CM (see Theorem \ref{Theorem: GT-varieties are aCM})
we have:

\begin{corollary}\label{Corollary: mu_d, codim, degree} \begin{itemize}
\item[(i)] $\mu_{d} = \frac{d + \theta(a,b,d) + 2}{2}$ and $X_{d} \subset \PP^{\mu_{d}-1}$ is a projective surface of degree $deg(X_{d}) = d$ and codimension $codim(X_{d}) = \frac{d + \theta(a,b,d) - 4}{2}$. If $d$ is prime, $\mu_{d} = \frac{d + 5}{2}$ and $codim(X_{d}) = \frac{d-1}{2}$.
\item[(ii)] $S/I(X_{d})$ is a level ring of Cohen-Macaulay type $\frac{d-\theta(a,b,d)+2}{2}$ with Castelnuovo Mumford regularity $reg(X_{d}) = 3$.
\end{itemize}
\end{corollary}

The information on the Hilbert function $HF(X_{d},z)$ and the regularity allow us to determine a minimal graded free $S$-resolution of any $GT$-surface $X_d$. We set $c = codim(X_{d})$ and $h = deg(X_{d}) - c-2 = \frac{d-\theta(a,b,d)+2}{2} - 1$.

\begin{theorem}\label{Theorem: minimal free resolution} \begin{itemize}
\item[(i)] If $\theta(a,b,d) = 3$, then a minimal graded free $S$-resolution of $S/I(X_{d})$  is
$$0 \to S^{b_{c,2}}(-c-2) \to \oplus_{i=1}^{2} S^{b_{c-1},i}(-c-i+1) \to \oplus_{i=1,2} S^{b_{c-2,i}}(-c-i+2)$$
$$\to \cdots \to \oplus_{i=1,2} S^{b_{1,i}}(-1-i) \to S \to S/I(X_{d}) \to 0,$$
where
$$b_{l,i} = \left\{ \begin{array}{lcl}
l\binom{c}{l+1} & & \text{if} \;\; 1 \leq l \leq c-1, \; i = 1\\
l\binom{c}{l}   & & \text{if} \;\; 1 \leq l \leq c, \; i = 2.
\end{array} \right.$$
\item[(ii)] If $\theta(a,b,d) \geq 4$, a minimal graded free $S$-resolution of $S/I(X_{d})$ is
$$0 \to S^{b_{c,2}}(-c-2) \to \oplus_{i=1}^{2} S^{b_{c-1},i}(-c-i+1) \to
\oplus_{i=1,2} S^{b_{c-2,i}}(-c-i+2)$$
$$\to \cdots \to \oplus_{i=1,2} S^{b_{c-h,i}}(-c-i+h) \to S^{b_{c-h-1,1}}(-c+h)$$
$$\to \cdots \to
S^{b_{1,1}}(-2) \to S \to S/I(X_{d}) \to 0,$$
where
$$b_{l,i} = \left\{ \begin{array}{lcl}
l\binom{c}{l+1} + (c-h-l)\binom{c}{l-1}& & \text{if} \;\; 1 \leq l \leq c-h-1, \; i = 1\\
l\binom{c}{l+1}   & & \text{if} \;\; c-h \leq l \leq c-1, \; i = 1\\
(l-c+h+1)\binom{c}{l} & & \text{if}\;\; c-h \leq l \leq c, \; i= 2.
\end{array} \right.$$
\end{itemize}
\end{theorem}

\begin{proof} (i) The hypothesis $\theta(a,b,d) = 3$ implies $deg(X_{d}) = d = 2c+1$. We are in the assumptions of \cite[Corollary 3.4(i)]{Yanagawa}, from which the result follows.

\noindent (ii) If $\theta(a,b,d) \geq 4$, we have that
$deg(X_{d}) = d \leq 2c$. We show that if $d \geq 9$, then
$deg(X_{d}) = d \geq c+3$, and in this case the result follows from \cite[Corollary 3.4(ii)]{Yanagawa}. The remaining cases associated to $d = 4,6$ and $8$ have been checked computationally in Example \ref{Example: resolution surfaces} using the software Macaulay2 (\cite{M}). The inequality $d \geq c+3$ is equivalent to
$\theta(a,b,d) + 2  = (a,d) + (\lambda,d') + (\lambda-(a,d),d') + 2 \leq d$. Next we see that it holds for each $d \geq 9$.
It is straightforward to see that $d = (a,d)(\lambda,d')(\lambda-(a,d),d')\overline{d}$ with $\overline{d} \geq 1$. Now consider the system of inequalities $\alpha \beta \gamma \overline{d} - \alpha - \beta - \gamma - 2  < 0$ with $\alpha,\beta,\gamma \geq 1$. There are no integer solutions for $\overline{d} \geq 5$. For $1 \leq \overline{d} \leq 4$, it is easy to see that $d\leq 8$.
\end{proof}

\begin{remark} \rm \iffalse It is worth mentioning that the Betti numbers of $X_{d}$ only depends on $d$ and the linear representation $\langle M_{d;0,a,b} \rangle$ of $\ZZ/d\ZZ$. \fi Fix $d\geq 3$ and let $X_{d}$ and $X_{d}'$ be $GT$-surfaces with groups $\Lambda = \langle M_{d;0,a,b} \rangle$ and $\Lambda' = \langle M_{d;0,a',b'} \rangle \subset GL(3,k)$, respectively. If $\theta(a,b,d) = \theta(a',b',d)$, then $S/I(X_{d})$ and $S/I(X_{d}')$ have the same Betti numbers.
\end{remark}

A consequence of Theorem \ref{Theorem: minimal free resolution} is the following.

\begin{corollary}\label{Corollary: number of generators} \begin{itemize}
\item[(i)] If $\theta(a,b,d) = 3$, then $I(X_{d})$ is minimally generated by $\binom{\mu_{d}-3}{2}$ quadrics and $\mu_{d}-3$ cubics.
\item[(ii)] If $\theta(a,b,d) \geq 4$, then $I(X_{d})$ is minimally generated by $\binom{\mu_{d}-3}{2} +2(\mu_{d}-3)-d+1$ quadrics.
\end{itemize}
\end{corollary}

\begin{remark}\rm  With Theorem \ref{Theorem: minimal free resolution} we recover \cite[Theorem 7.2]{MM-R}, where the authors determine a minimal graded free resolution of the $GT$-surface with group $\Lambda = \langle M_{d;0,1,2} \rangle \subset GL(3,k)$.
\end{remark}

We end this section showing the shape of a minimal graded free resolution of the coordinate ring of all $GT$-surfaces $X_{d}$ for $d = 4,6,8$. All the computations have been made with the software Macaulay2 (\cite{M}).

\begin{example} \rm \label{Example: resolution surfaces}
(i)  Fix $d = 4$ and let $X_{4}$ be a $GT$-surface with group $\Lambda = \langle M_{4;0,a,b} \rangle \subset GL(3,k)$. For all integers $0 < a < b < 4$ with $GCD(a,b,d) = 1$, we have that $\theta(a,b,4) = 4$. Let $S = k[w_1,\hdots,w_5]$: in any case a minimal graded free $S$-resolution of $S/I(X_{4})$ is of the form
$$0 \to S(-4) \to S^{2}(-2) \to A \to S/I(X_{4}) \to 0,$$
i.e., $X_{4} \subset \PP^{4}$ is a complete intersection of $2$ quadrics.

(ii) Fix $d = 6$  and let $X_{6}$ be a $GT$-surface with group $\Lambda = \langle M_{6;0,a,b} \rangle GL(3,k)$. We have:
$$\theta(a,b,6) = \left\{\begin{array}{lclcl}
4 & \text{if} & a = 1 & \text{and} & b = 2,5; \; \text{or}\\
  &           & a = 4 & \text{and} & b = 5.\\
5 & \text{otherwise}. & & &
\end{array}\right.$$

Let $S = k[w_1,\hdots,w_6]$ and $\overline{S} = k[w_1,\hdots,w_7]$. A minimal graded free $S$-resolution of $S/I(X_{6})$ with $\theta(a,b,6) = 4$ has the shape:
$$0 \to S^{2}(-5) \to S^{3}(-4)\oplus S^{2}(-3) \to
S^{4}(-2) \to S \to S/I(X_{6}).$$
A minimal graded free $\overline{S}$-resolution of $S/I(X_{6})$ with $\theta(a,b,6) = 5$ has the shape:
$$0 \to \overline{S}(-6) \to \overline{S}^{9}(-4) \to \overline{S}^{16}(-3) \to \overline{S}^{9}(-2) \to \overline{S} \to \overline{S}/I(X_{6}) \to 0.$$
In this case, $X_{6}$ is an arithmetically Gorenstein surface of $\PP^{6}$.

(iii) Fix $d = 8$ and let $X_{8}$ be a $GT$-surface with group $\Lambda = \langle M_{8,0,a,b} \rangle$. We have:
$$\theta(a,b,8) = \left\{\begin{array}{lclcl}
5 & \text{if} & a = 1 & \text{and} & b = 4,5; \; \text{or}\\
  &           & a = 3 & \text{and} & b = 4,7; \; \text{or}\\
  &           & a = 4\\
4 & \text{otherwise}. & & &
\end{array}\right.$$
Let $S = k[w_1,\hdots,w_8]$ and $\overline{S} = k[w_1,\hdots,w_7]$. As in the previous case, we obtain the following resolutions:

$$0 \to S^{2}(-7) \to S^{5}(-6)\oplus S^{4}(-5) \to S^{25}(-4) \to S^{30}(-3) \to S^{13}(-2) \to S \to S/I(X_{8}) \to 0,$$
$$0 \to \overline{S}^{3}(-6) \to \overline{S}^{8}(-5) \oplus \overline{S}^{3}(-4) \to \overline{S}^{6}(-4)\oplus \overline{S}^{8}(-3) \to\overline{S}^{7}(-2) \to \overline{S} \to \overline{S}/I(X_{8}) \to 0.$$

\end{example}

%%%%%%%%%%%%%%%%%%%%%%%%%%%%%%%%%%%%%%%%%%%%%
\section{A new family of aCM surfaces parameterized by monomial Togliatti systems}
\label{aCM togliatti surfaces}

Let $n,d$ be positive integers and fix $e$, a $d$th primitive root of $1$. We denote by $\Gamma \subset GL(n+1,k)$ the finite diagonal group of order $d$ generated by $M_{d;1,\hdots,1}:=diag(e,\hdots,e)$. The {\em Veronese variety} $V_{n,d} \subset \PP^{\binom{n+d-1}{n-1}-1}$ is the projective variety whose homogeneous coordinate ring is the ring of invariants $R^{\Gamma}$. The set $\mathcal{M}_{n,d} \subset R$ of all monomials of degree $d$ is a $k$-algebra basis of $R^{\Gamma}$. By a {\em monomial projection of $V_{n,d}$}, we mean a projective variety given parameterically by a subset of $\mathcal{M}_{n,d}$. In \cite{Grob},  Gr\"{o}bner  posed the problem of determining which monomial projections of  Veronese varieties are aCM. Since then, there have been many efforts to solve this still open problem, see for instance \cite{Hoa}, \cite{Trung} and \cite{Trung1}. 
In Section \ref{the aCM of GT-varieties}, we proved that all $GT$-varieties with finite linear diagonal cyclic group are aCM. However, not all surfaces parameterized by monomial Togliatti systems are aCM.  For instance, the Togliatti system  $I = \{x_{0}^{5}, x_{1}^{5}, x_{2}^{5}, x_{0}^{3}x_{1}x_{2}, x_{0}^{2}x_{1}^{2}x_{2}, x_{0}x_{1}^{3}x_{2}\} \subset k[x_{0},x_{1},x_{2}]$ gives rise to a non aCM surface $X:=\varphi_{I}(\PP^{2}) \subset \PP^{5}$. Indeed, we have checked with the software Macaulay2, \cite{M}, that $codim(X) = 3 < pd(S/I(X)) = 4$. 

It is then natural to pose the following problem:
\begin{problem} To determine whether a monomial projection of $V_{2,d}$, corresponding to a monomial Togliatti system, is aCM.
\end{problem}

In this section, we prove the arithmetic Cohen-Macaulayness of a new family of surfaces parameterized by monomial Togliatti systems: their coordinate ring is not the ring of invariants of any finite linear group. Nevertheless, their construction is rather naturally related to GT-systems. We denote $R = k[x_{0},x_{1},x_{2}]$.

\begin{definition} \rm  We define the semigroup $H_{3}:= \langle (3,0,0), (0,3,0), (0,0,3), (1,1,1) \rangle \subset \ZZ^3_{\geq 0}$. Set $m = (1,1,1)$. Inductively for $t \geq 2$, we define $H_{3t} := \langle (3t,0,0), (0,3t,0), (0,0,3t), m + H_{3(t-1)} \rangle$, where $m + H_{3(t-1)} = \{m+h \, \mid \, h \in H_{3(t-1)}\}$.
\end{definition}

Let us illustrate the above definition  with the following three examples.

\begin{example} \rm \label{Example. first examples of H3t} (i)
$H_{6} = \langle (6,0,0), (0,6,0), (0,0,6), (4,1,1), (1,4,1), (1,1,4), (2,2,2) \rangle$.

(ii) $H_{9} = \langle (9,0,0), (0,9,0), (0,0,9), (7,1,1), (1,7,1), (1,1,7), (5,2,2), (2,5,2), (2,2,5),$

$ (3,3,3) \rangle$.

(iii) $H_{12} = \langle (12,0,0), (0,12,0), (0,0,12),
(10,1,1), (1,10,1), (1,1,10), (8,2,2), (2,8,2),$

$ (2,2,8),(6,3,3), (3,6,3), (3,3,6), (4,4,4) \rangle$.
\end{example}

We denote by $J_{3t} \subset R$ the monomial artinian ideal associated to $H_{3t}$. All ideals $J_{3t}$ have $\mu_{3t} = 3t + 1$ generators. It is easy to check by induction that they are Togliatti systems. Indeed, the first ideal $J_{3}$ is of course the monomial $GT$-system (\ref{3togliatti}) with group $\langle M_{3; 0,1,2} \rangle \subset GL(3,k)$. On the other hand, for any $t$, $J_{3t}=(x_0^{3t}, x_1^{3t}, x_2^{3t},x_0x_1x_2J_{3t-1})$. 

By Theorem \ref{Theorem: GT-varieties are aCM}, $k[H_{3}]$ is CM. Notwithstanding, for $t > 1$ the semigroups $H_{3t}$ are not normal and $\;k[H_{3t}]$ are not rings of invariants of finite linear groups. 
For $t > 1$, $H_{3t}$ is not normal since $m \in \overline{H_{3t}}$, the saturation of $H_{3t}$ (see Definition \ref{Definition: Normal semigroup}), and $m \notin H_{3t}$. To check the second assertion, assume by contradiction  that $k[H_{3t}]$ is the ring of invariants of a finite group $G \subset GL(3,k)$, and let $\rho: R \to R^{G}$ be the Reynolds operator. We have that for all $t > 1$, $(3,3(t-1),0) \notin H_{3t}$ (see Lemma \ref{Lemma: Elements with only one zero component}), or equivalently  $x_0^3x_1^{3(t-1)} \notin R^{G}$. We observe that $(3,3(t-1),0) + tm$ can be written as $[(t-1)m + (3,0,0)] + [m + (0,3(t-1),0)] \in H_{3t}$. So $x_0^tx_1^tx_2^t \cdot x_0^3x_1^{3(t-1)} \in R^{G}$ and we have $\rho(x_0^{t}x_1^{t}x_2^{t}\cdot x_0^3x_1^{3(t-1)}) = x_0^tx_1^tx_2^t \cdot \rho(x_0^3x_1^{3(t-1)}) = x_0^{t}x_1^{t}x_2^{t}\cdot  x_0^{3}x_1^{3(t-1)}$. Therefore $\rho(x_0^3x_1^{3(t-1)}) = x_0^{3}x_1^{3(t-1)}$ and we get a contradiction.

\medskip

Our goal is to prove that all $k[H_{3t}]$ are CM rings. To this end, we want to apply Theorem \ref{Theorem:Criterion aCM-Trung}. But first we need some preparation. We fix $t > 1$ and we put $f_{1} = (3t,0,0), f_{2} = (0,3t,0), f_{3} = (0,0,3t)$.

\begin{remark} \rm \label{Remark: components of generators} (i) Notice that $f_{1},f_{2}$ and $f_{3}$ are $\Q$-linearly independent and $(3t)H_{3t} \subset \langle f_{1},f_{2},f_{3} \rangle$.

(ii) By construction $H_{3t} \subset H_{3}$, so $\overline{H_{3t}} \subset H_{3}$. This means that for all $u = (a_{1},a_{2},a_{3}) \in H_{3t}$ there exist $f \geq 1$ and $r \in \{0,\hdots, 2tf\}$  such that $u$ is a solution of the system:
$$(*) = \begin{cases}
a_{1} + a_{2} + a_{3} =  3ft \\
a_{2} + 2a_{3} = 3r.
\end{cases}$$
The converse is not true: $(3,3(t-1),0) \notin H_{3t}$ but it belongs to $H_{3}$.

(iii) All generators of $H_{3t}$ different from $f_{1},f_{2},f_{3}$  have all three components different from $0$.
\end{remark}

\begin{remark} \rm \label{Remark: expression of generator} By construction, we can describe $$H_{3t} = \{u = A_{1}f_{1} + A_{2}f_{2} + A_{3}f_{3} + \sum_{j=1}^{3(t-1)+1} A_{j+3} (m + h_{j})\}\subset \ZZ_{\geq 0}^{3},$$ where $A_{i} \in \ZZ_{\geq 0}$ for $ i = 1, \hdots, 3t+1$ and $h_{j}$ is a generator of $H_{3(t-1)}$, for $j = 1, \hdots, 3(t-1)+1.$ Notice that a generator $h = (a_{1},a_{2},a_{3})$ of $H_{3t}$ different from $f_{1},f_{2},f_{3}$ can be expressed as $sm + h'$, where $0 < s = min\{a_{1},a_{2},a_{3}\} \leq t$ and $h' \in \{(3(t-s),0,0),(0,3(t-s),0), (0,0,3(t-s))\}$.
\end{remark}

We give a couple of examples.

\begin{example} \rm (i) Consider $H_{6}$. We have:
$(4,1,1) = m + (3,0,0)$, $(1,4,1) = m + (0,3,0),$ $(1,1,4) = (1,1,1) + (0,0,3)$ and $(2,2,2) = 2m$.

(ii) Consider $H_{9}$. We have:
$(7,1,1)= m + (6,0,0),$ $(1,7,1) = m + (0,6,0),$
$(1,1,7) = m + (0,0,6),$ $(5,2,2) = 2m + (3,0,0),$
$(2,5,2) = 2m + (0,3,0),$  $(2,2,5) = 2m + (0,0,3)$ and $(3,3,3) = 3m$.
\end{example}

Any $u \in H_{3t}$ represents a monomial of degree a multiple of $3t$, namely $(3t)f$. For any representation
$u = A_{1}f_{1} + A_{2}f_{2} + A_{3}f_{3} + \sum_{j=1}^{3(t-1)+1} A_{j+3} (m + h_{j})$ in $H_{3t}$, it holds that $\sum_{i=1}^{3t+1} A_{i} = f$.

\begin{lemma} \label{Lemma: Elements with only one zero component} Let $w = (a_{1},a_{2},a_{3}) \in H_{3}$ be such that $a_{i},a_{j} \neq 0$ and $a_{k} = 0$, for $\{i,j,k\} = \{1,2,3\}$. Then $w \in H_{3t}$ if and only if $a_{i}$ and $a_{j}$ are multiples of $3t$.
\end{lemma}

\begin{proof} We can assume $(i,j,k) = (1,2,3)$. If $w = (a_{1},a_{2},0) \in H_{3t}$, then $w$ cannot be generated in $H_{3t}$ by any element belonging to $m + H_{3(t-1)}$. So we obtain $w = A_1f_{1} + A_2f_{2}$ with $a_1 = 3tA_1$ and $a_2 = 3tA_2$. Conversely, $w = (3tA_1 ,3tA_2,0) \in H_{3t}$ for all integers $ A_1,A_2\geq 0.$
\end{proof}

\begin{corollary}\label{Corollary: Only two nonzero components} If $w\in H_3$ is as in Lemma \ref{Lemma: Elements with only one zero component}, then either $w \in H_{3t}$ or $w + f_{i}, w + f_{j} \notin H_{3t}$.
\end{corollary}

\begin{remark} \rm \label{Remark: Only one nonzero component} If $w = (a_{1},a_{2},a_{3}) \in H_{3t}$ only has one nonzero component, namely $a_{i}$, then $w = A_if_i$, where $a_i = 3tA_i$.
\end{remark}

We are now ready to prove the main theorem of this section.

\begin{theorem}\label{Theorem: Main Theorem} For any $t\geq 1$, $k[H_{3t}]$ is CM.

\end{theorem}
\begin{proof} By Theorem \ref{Theorem:Criterion aCM-Trung}, it is enough to prove that $H^{1} = \{w \in \overline{H_{3t}} \, \mid \,
w + f_{i}, w+ f_{j} \in H_{3t} \; \text{for some} \; i, j \in \{1,2,3\}, i\neq j \}$ is contained in  $H_{3t}$.  We {\em claim} that this inclusion is a consequence of the following  condition: 

{\em Condition $(*)$:} if $w = (a_{1},a_{2},a_{3}) \in H_{3}$ is such that $a_{1}a_{2}a_{3} \neq 0$ and  $w + f_{i} \in H_{3t}$ for some $i \in \{1,2,3\}$, then either $w \in H_{3t}$ or $w + f_{j}, w + f_{k} \notin H_{3t}$ for $\{i,j,k\} = \{1,2,3\}$. 
\medskip

{\em Proof of the claim}. We have already shown the same statement for elements $w$ with $a_{1}a_{2}a_{3} = 0$ in Corollary \ref{Corollary: Only two nonzero components} and Remark \ref{Remark: Only one nonzero component}. Since $H^{1} \subset \overline{H_{3t}} \subset H_{3}$, an element $w \in H^{1}$ satisfying $w + f_{j}, w + f_{k} \in H_{3t}$, for some $j,k \in \{1,2,3\}$ such that $j \neq k$, belongs to $H_{3t}$.  This proves the claim.
\medskip

{\em Proof of Condition $(*)$}. We can assume $(i,j,k) = (1,2,3)$. Set $w + f_{1} = A_{1}f_{1} + A_{2}f_{2} + A_{3}f_{3} + \sum_{j} A_{j+3} (m + h_{j}) \in H_{3t}$. We may assume that $A_{1} = 0$, otherwise the result is trivial. We observe the following. Let $u = m + h_{j} = s_{j}m + (3(t-s_{j}),0,0)$ and $v = m + h_{i} = s_{i}m + (3(t-s_{i}),0,0)$, with $s_{j},s_{i} > 0$, be two generators of $H_{3t}$. Therefore we can write $u + v = [(s_{j}-1)m + (3(t-s_{j}+1),0,0)]+
[(s_{i}+1)m + [(3(t-s_{i}-1),0,0)].$ %The same argument works 
Similarly if we replace $h_{j}$, $h_{i}$ by $(0,3(t-s_{j}),0), (0,3(t-s_{i}),0)$ or  $(0,0,3(t-s_{j}))$, $(0,0,3(t-s_{i}))$ respectively. So after doing suitable transformations on the summands of $w + f_{1}$, we reduce it to one of the following forms.

\vspace{0.2cm}
\noindent \underline{Case 1}: $w + f_{1} = A_{2}f_{2} + A_{3}f_{3} + [s_{1}m
+ (3(t-s_{1}),0,0)] + [s_{2}m + (0,3(t-s_{2}),0)] + [s_{3}m + (0,0,3(t-s_{3}))]$ with $0 < s_{1} < t$.
Since $s_{1} + s_{2} + s_{3} + 3(t-s_{1}) = 3t + a_{1}$, we have $0 \leq s_{2},s_{3} < t$, where $s_{2} > 0$ or $s_{3} > 0$. Let us assume that $s_{2},s_{3} > 0$, the other cases follow in the same way up to minor modifications. By hypothesis, $w + f_{1}$ can be written as a sum of $A_{2} + A_{3} + 3$ generators of $H_{3t}$.
The first component of $w + f_{1}$ corresponds to  $a_{1} + 3t = s_{1} + 3(t-s_{1}) + s_{2} + s_{3}$, so $a_{1} = s_{2} + s_{3} - 2s_{1}$.
Notice that $w = (s_{2} + s_{3} -2s_{1}, s_{1} + s_{2} + s_{3} + A_{2}3t + 3(t-s_{2}), s_{1} + s_{2} + s_{3} + A_{3}3t + 3(t-s_{3}))$. If $s_{2},s_{3} \geq s_{1}$, we have
$w = A_{2}f_{2} + A_{3}f_{3} + [(s_{2}-s_{1})m + (0,3(t-s_{2}+s_{1}),0)] + [(s_{3}-s_{1})m + (0,0,3(t-s_{3}+s_{1}))]$.
Indeed, $s_{1}+s_{2}+s_{3} = s_{2}-s_{1} + s_{3} - s_{1} + 3s_{1}$, hence $w \in H_{3t}$. Otherwise, suppose for instance that $s_{2} < s_{1}$ and write
\begin{equation}
w = (s_{2}+s_{3}-2s_{1})m +
(0, A_{2}3t + 3t - 3s_{2} + 3s_{1}, A_{3}3t + 3t - 3s_{3} + 3s_{1}).
\end{equation}
If $w \in H_{3t}$, then $w$ is a sum of $A_{2} + A_{3} + 2$ generators of $H_{3t}$. We observe that $A_{2}3t + 3t - 3s_{2} + 3s_{1} > (A_{2}+1)3t$, $A_{3}3t + 3t - 3s_{3} + 3s_{1} > A_{3}3t$ and $s_{2} + s_{3} - 2s_{1} < s_{3} < t$. This means that we can write $w$ as a sum of at least $A_{2} + 2$ generators of  type $sm + (0,3(t-s),0)$ plus at least $A_{3} + 1$ generators of type $sm + (0,0,3(t-s))$, where all $s < t$. Indeed, since $a_{1} = s_{2} + s_{3} - 2s_{1} < t$, a generator in $w$ cannot be of the form $tm$, otherwise $w + f_{1}$ does. If this was the case,  such generator would be either $f_{2}$, or $f_{3}$, or it would correspond to $sm + (0,3(t-s),0)$ or $sm + (0,0,3(t-s))$ with $0 < s < t$.  But this is a contradiction, because that would give rise to an expression of $w$ with at least $A_{2}+A_{3}+3$ summands (see Remark \ref{Remark: components of generators}(3)). Performing the same kind of arguments, we see that $w+f_2,w+f_3 \notin H_{3t}$. 
The case $s_{3} < s_{1}$ is analogous.

\vspace{0.2cm}
\noindent \underline{Case 2:} $w + f_{1} = A_{2}f_{2} + A_{3}f_{3} + tm + [s_{1}m + (3(t-s_{1}),0,0)] + [s_{2}m + (0,3(t-s_{2}),0)] + [s_{3}m + (0,0,3(t-s_{3}))]$, where $s_{1} > 0$ and some $s_i > 0, \; i = 2,3$. We assume $s_{2},s_{3} > 0$ for simplicity. By hypothesis, $w + f_{2}$ is a sum of $A_{2} + A_{3} + 4$ generators of $H_{3t}$. If $s_{2} > s_{1}$ (respectively  $s_{3} > s_{1}$),
 $$w = A_{2}f_{2} + A_{3}f_{3} + (t-s_{1})m + (0,3s_{1},0) +
 s_{2}m + (0,3(t-s_{2}),0) + (s_{3}-s_{1})m + (0,0,3(t-s_{3}+s_{1})),$$
 hence $w \in H_{3t}$. We see that if $s_2,s_3 < s_1$, then  $w \notin H_{3t}$. If not, $w$ can be written as a sum of $A_{2} + A_{3} + 3$ generators and we have:
$$w = m(t + s_{2} + s_{3} - 2s_{1})
+ (0, 3tA_{2} + 3t - 3s_{2} + 3s_{1}, 3tA_{3} + 3t - 3s_{3} + 3s_{1}).$$
Notice that $t + s_{2} + s_{3} - 2s_{1} < t$, $ 3tA_{2} + 3t - 3s_{2} + 3s_{1} > (A_{2} + 1)3t$ and $3tA_{3} + 3t - 3s_{3} + 3s_{1} > (A_{3} + 1)3t$. So, $w$ is a sum of at least $A_{2} + A_{3} + 4$ generators of $H_{3t}$. Arguing in a similar way, we also obtain that $w + f_{2}, w + f_{3} \notin H_{3t}$.

\vspace{0.2cm}
\noindent \underline{Case 3:} $w + f_{1} = A_{2}f_{2} + A_{3}f_{3} + 2tm + [s_{1}m + (3(t-s_{1}), 0,0)] + [s_{2}m + (0,3(t-s_{2}),0) + s_{3}m + (0,0,3(t-s_{3}))]$. Here the situation is slightly different. If $s_{1} > 0$, then $w \in H_{3t}$. Indeed, $w = A_{2}f_{2} + A_{3}f_{3} + [(t-s_{1})m + (0,3(t-s_{1}),0)] + [(t-s_{1})m + (0,0,3(t-s_{1}))] + [s_{2}m + (0,3(t-s_{2}),0)] + [s_{3}m + (0,0,3(t-s_{3}))].$ So we suppose $s_{1} = 0$, in which case $s_{2},s_{3} > 0$ and we have:
$$w = (s_{2} + s_{3} - t)m + (0,3tA_{2} + 3t + 3t - 3s_{2}, 3tA_{3} + 3t + 3t - 3s_{3}),$$
 with $s_{2} + s_{3} - t < t$, $3tA_{2} + 3t + 3t - 3s_{2} >
 (A_{2} + 1)3t$ and $3tA_{3} + 3t + 3t - 3s_{3} > (A_{3} + 1)3t$.
If $w \in H_{3t}$, then  it should be written as a sum of at least $A_{2} + A_{3} + 4$ generators, which is a contradiction. Performing the same arguments we also obtain $w + f_{2}, w + f_{3} \notin H_{3t}$.

 \vspace{0.2cm}
 \noindent \underline{Case 4:} $w+f_{1} =
 A_{2}f_{2} + A_{3}f_{3} + K(tm) + [s_{1}m + (3(t-s_{1}),0,0)] + [s_{2}m + (0,3(t-s_{2}),0)] + [s_{3}m + (0,0,3(t-s_{3}))]$, with $K \geq 3$.  We always have $w \in H_{3t}$, indeed $tm + tm + tm = f_{1} + f_{2} + f_{3}$.

\medskip
This proves {\em Condition $(*)$} and the theorem follows.
\end{proof}

Let us see how Theorem \ref{Theorem: Main Theorem} works in $k[H_{6}]$.
\begin{example}\rm

\noindent \underline{Case 1.} The only possibility is $w + f_{1} = A_{2}(0,6,0) + A_{3}(0,0,6) + [(1,1,1) + (3,0,0)] + [(1,1,1) + (0,3,0)] + [(1,1,1) + (0,0,3)]$, where necessarily $a_{1} = 0$. For simplicity we set $A_{2} = A_{3} = 0$. If $s_{1},s_{2} > 0$, then $w = (0, 1 + 4 + 1,  1 + 1 + 4) = f_{2} + f_{3} \in H_{6}$.

\vspace{0.1cm}
\noindent \underline{Case 2.} We consider $w + f_{1} = (2,2,2) + [(1,1,1) + (3,0,0)] + [(1,1,1) + (0,3,0)] + [(1,1,1) + (0,0,3)]$, with $s_{1} = s_{2} = s_{3} = 1$. Then we have: $w = (2,2,2) + (0, 2 + 4, 2 + 4) = [m + (0,3,0)] + [m + (0,0,3)] \in H_{6}$.

\vspace{0.1cm}
\noindent \underline{Case 3.} We consider $w + f_{1}
= (2,2,2) + (2,2,2) + [(1,1,1) + (0,3,0)] + [(1,1,1) + (0,0,3)]$,
with $a_{1} = 0$. Then we have: $w = (0,9,9), w + (0,6,0) = (0,15,9), w + (0,0,6) = (0,9,15) \notin H_{6}$.
\end{example}

Fix an integer $k\geq 1$. For each integer $t' \geq 0$, we define
$H_{3(1 + t'k)}^{k} := \langle (3(1+t'k),0,0), (0,3(1+t'k),0), (0,0,3(1+t'k)), km + H_{3(1 + (t'-1)k)}^{k}\rangle \subset \ZZ_{\geq 0}^{3}$. We have:

\begin{corollary} $k[H_{3(1+kt')}^{k}]$ is CM for all integers $k \geq 1$ and $t' \geq 0$.
\end{corollary}
\begin{proof} It follows from the same proof as Theorem \ref{Theorem: Main Theorem} replacing $m$ by $km$.
\end{proof}

\begin{remark} \rm \begin{enumerate}
\item[(i)] $H_{3(1 + t'k)}^{k}$ is generated by $3(t'+1) + 1$ elements in $\ZZ^{3}$.
\item[(ii)] Our initial family $H_{3t}$ can be rewritten as
$H_{3(1+t')}^{1}$ for $t' \geq 0$.
\end{enumerate}
\end{remark}

%%%%%%%%%%%%%%%%%%%%%%%%%%%%%%%%%%%%%%
%%%%%%%%%%%%%%%%%%%%%%%%%%%%%%%%%%%%%%%%%%%%%%%%%%

\end{document}